\documentclass[11pt]{amsart}
\pdfoutput=1
\usepackage{amssymb,amsmath,amscd,url}

\usepackage{graphicx}

\newtheorem{theorem}{Theorem}[section] 
\newtheorem{lemma}[theorem]{Lemma}     
\newtheorem{corollary}[theorem]{Corollary}
\newtheorem*{theoremn}{Theorem \ref{duality}}

\newtheorem*{corollaryn}{Corollary \ref{affine theorem}}
\newtheorem{proposition}[theorem]{Proposition}
\theoremstyle{remark}
\newtheorem{remark}[theorem]{Remark}
\theoremstyle{definition}
\newtheorem{definition}[theorem]{Definition}
\usepackage{amsfonts,url}
\newcommand{\cl}{\mathcal{C}\ell}

\newcommand{\la}{\langle}
\newcommand{\ra}{\rangle}
\newcommand{\ol}{\overline}
\newcommand{\e}{\mathbf{e}}
\newcommand{\bQ}{\mathbf{Q}}
\newcommand{\bF}{\mathbf{F}}
\newcommand{\eps}{\varepsilon}

\DeclareMathOperator{\coker}{coker}

\newcommand{\tensor}{\hspace{0.1em}\widehat{\otimes}\hspace{0.15em}}

\newcommand{\R}{\mathbb{R}}

\newcommand{\Z}{\mathbb{Z}}
\newcommand{\C}{\mathbb{C}}

\def\GL{{\rm GL}}
\def\O{{\rm O}}

\def\cE{{\mathcal E}}
\def\cF{{\mathcal F}}

\def\Ext{{\rm Ext}}

\def\PSU{{\rm PSU}}

\def\SO{{\rm SO}}

\def\SU{{\rm SU}}
\def\GL{{\rm GL}}

\def\U{{\rm U}}

\def\Hom{{\rm Hom}}

\def\cB{{\mathcal B}}

\def\cS{{\mathcal S}}

\def\cU{{\mathcal U}}

\def\cZ{{\mathcal Z}}

\def\bH{{\mathbf H}}
\def\bL{{\mathbf L}}
\def\bS{{\mathbf S}}
\def\bX{{\mathbf X}}
\def\fb{{\mathfrak b}}
\def\fa{{\mathfrak a}}

\def\fx{{\mathfrak x}}
\def\fy{{\mathfrak y}}

\def\ft{{\mathfrak t}}
\def\fu{{\mathfrak u}}

\def\fG{{\mathfrak G}}

\def\ft{{\mathfrak t}}

\def\Sp{{\rm Sp}}

\def\U{{\mathbf U}}

\title[Extended affine Weyl groups]{Extended Affine Weyl groups, the Baum-Connes correspondence and Langlands duality}

\author{Graham A. Niblo, Roger Plymen and Nick Wright}
\address{Mathematical Sciences, University of Southampton, SO17 1BJ,  England}
\email{g.a.niblo@soton.ac.uk, r.j.plymen@soton.ac.uk, wright@soton.ac.uk}

\begin{document}
\begin{abstract}  In this paper we consider the Baum-Connes correspondence for the affine and extended affine Weyl groups of a compact connected semisimple Lie group. We show that the Baum-Connes correspondence in this context arises from Langlands duality for the Lie group.
\end{abstract}

\maketitle

\tableofcontents

\section{Introduction}

Throughout this paper $G$ will denote a compact connected semisimple Lie group with maximal torus $T$, whose Lie algebra will be denoted $\ft$. We will examine the Baum-Connes correspondence in the context of affine and extended affine Weyl groups associated with $G$, which can be realised as groups of affine isometries of the Lie algebra $\ft$. The assembly map is an isomorphism in this context, cf.\ \cite{HK}. The domain of the assembly map for the extended affine Weyl group $W_a'$ is $KK^*_{W_a'}(C_0(\ft),\C)$ which we identify with $KK^*_{W}(C(T),\C)$ where $W$ is the Weyl group of $G$, see Remark \ref{LHS remark}.

The group $W_a'$ is the semidirect product $\Gamma\rtimes W$ where $\Gamma$ is the lattice of translations in $\ft$ corresponding to $\pi_1(T)$. The commutative algebra $C^*_r\Gamma$ is identified with $C(\widehat\Gamma)$ where $\widehat\Gamma$ denotes the Pontryagin dual of $\Gamma$. We identify the right hand side of the assembly map $KK^*(\C,C^*_rW_a')$ with the equivariant $KK$-group $KK_W(\C,C(\widehat\Gamma))$. Since $\widehat\Gamma$ is a torus of the same dimension as $T$ one might be tempted to think that the Baum-Connes correspondence is an isomorphism between the $W$-equivariant $K$-homology and $K$-theory of the torus $T$. While there is such an isomorphism (rationally) by the Universal Coefficients Theorem, this is not the Baum-Connes correspondence.

Although $\widehat\Gamma$ is a torus of the same dimension as $T$, there is in general no $W$-equivariant identification of the two tori. A very illustrative example is furnished by the Lie group $\SU_3$ whose extended affine Weyl group (which in this case is its affine Weyl group) is the $(3,3,3)$-triangle group acting on the plane. We find this an extremely useful way of visualising the Baum-Connes conjecture and examine it in detail in Section \ref{triangle section}.

We will see in this example that the left- and right-hand sides of the Baum-Connes correspondence look very different and the Baum-Connes isomorphism might almost appear coincidental in this case.  This `coincidence' however can be explained by $T$-duality between the tori $T$ and $\widehat \Gamma$. At the level of Lie groups $\widehat\Gamma$ is (equivariantly) the maximal torus $T^\vee$ of the Langlands dual $G^\vee$ of $G$. We construct a $W$-equivariant Poincar\'e duality in $K$-theory from $T$ to $T^\vee$, which can be viewed as providing a geometrical proof of the Baum-Connes correspondence in this context:

\begin{theorem}\label{the main theorem} Let $G$ be a compact connected semisimple Lie group, let $W'_a = W'_a(G)$ be the extended affine Weyl group attached to $G$.     Then we have the following commutative diagram:
$$\begin{CD}
KK^*_{W_a'}(C_0(\ft),\C) @>\mu>\phantom{\text{Poincar\'e-Langlands}}> KK^*(\C,C^*_rW_a')\\
@VV\cong V @A\cong A{\text{\begin{tabular}{c}Fourier-Pontryagin\\duality\end{tabular}}}A \\
KK^*_W(C(T),\C) @>\cong>{\text{\begin{tabular}{c}Poincar\'e-Langlands\\duality\end{tabular}}}> KK^*_W(\C,C(T^\vee))\\
\end{CD}$$
where $\mu$ is the Baum-Connes assembly map.
\end{theorem}

The duality between $G$ and $G^\vee$ is further amplified by the following theorem.
\newcommand{\dualthm}{Let $G$ be a compact connected semisimple Lie group and $G^\vee$ its Langlands dual. Let $W_a'(G)$, $W_a'(G^\vee)$ denote the extended affine Weyl groups of $G$ and $G^\vee$ respectively. Then  there is a rational isomorphism
$$K_*(C^*_r(W_a'(G)))\cong K_*(C^*_r(W_a'(G^\vee))).$$}
\begin{theorem}\label{duality}
\dualthm
\end{theorem}
\medskip

\newcommand{\affinecor}{Let $W_a(G)$, $W_a(G^\vee)$ be the affine Weyl groups of $G,G^\vee$. If $G$ is of adjoint type then rationally
$$K_*(C^*_r(W_a'(G)))\cong K_*(C^*_r(W_a(G^\vee))).$$
If additionally $G$ is of type $A_n, D_n, E_6, E_7, E_8, F_4, G_2$ then rationally
$$K_*(C^*_r(W_a'(G)))\cong K_*(C^*_r(W_a(G))).$$}
\begin{corollary}\label{affine theorem}
\affinecor
\end{corollary}

The paper is structured as follows. In Section \ref{triangle section} we consider in detail our motivating example of the $(3,3,3)$-triangle group which arises as the affine Weyl group of $\SU_3$.  In Section \ref{langlands section} we recall the definition of Langlands duality for compact semisimple Lie groups via their complexifications. In particular we identify the reduced $C^*$-algebra of the nodal group for $T$ as the algebra of functions on the maximal torus $T^\vee$ of the Langlands dual group. In the main section, Section \ref{poincare section}, we construct our Poincar\'e duality in $KK$-theory between the algebras $C(T)$ and $C(T^\vee)$.
We recall the definitions of affine and extended affine Weyl groups in Section \ref{affine Weyl section} in preparation to prove, in Section \ref{the end section}, the main theorems stated above.

We would like to thank Maarten Solleveld for his helpful comments on the first version of this paper.

\section{The example of $\SU_3$: the $(3,3,3)$-triangle group}\label{triangle section}

For a compact connected semisimple Lie group $G$ the Weyl group $W$ is a finite Coxeter group which acts linearly on the Lie algebra $\ft$ of a maximal torus $T$. The extended affine Weyl group $W_a'$ of $G$ is the semidirect product of a $W$-invariant lattice $\Gamma$ in $\ft$ by $W$. The affine Weyl group $W_a$ is the semidirect product of a $W$-invariant sublattice $N$ of $\Gamma$ by the group $W$, where the index of $N$ in $\Gamma$ is the order of $\pi_1(G)$. In particular for $\SU_3$, which is simply connected, $N=\Gamma$ and $W_a=W_a'$. This group is the $(3,3,3)$-triangle group. 

For the quotient $\PSU_3=\SU_3/C_3$ the lattice $N$ remains the same as for $\SU_3$ while the lattice $\Gamma$ now contains $N$ as an index $3$ sublattice. In this section we will reserve the notation $\Gamma$ for the $\Gamma$-lattice of $\PSU_3$.

We illustrate the content of Theorems \ref{the main theorem},\ref{duality} and Corollary \ref{affine theorem} for $\SU_3$ and $\PSU_3$ by considering in detail the group $C^*$-algebras of the corresponding affine and extended affine Weyl groups. After posting the initial version of this paper it was drawn to our attention by Maarten Solleveld that the K-theory groups appearing in this section were computed in his thesis, \cite{S}, using similar methods but with slightly different exact sequences of algebras. We refer the reader to the thesis for a number of other interesting examples.

A maximal torus $T$ for $\SU_3$ is given by the diagonal matrices with entries $\alpha,\beta,\gamma$ in $\U:=\{\lambda\in\C: |\lambda|=1\}$ such that $\alpha\beta\gamma=1$. The exponential map allows us to identify the universal cover of $T$ with the tangent plane
$$\ft=\{(x,y,z)\in\R^3 : x+y+z=0\}.$$
Explicitly we use the map
$$(x,y,z) \mapsto \begin{pmatrix}e^{2\pi ix}\\&e^{2\pi iy}\\&&e^{2\pi iz}\end{pmatrix}$$
and we identify the torus with the quotient of this plane by the group
$$N=\{(a,b,c)\in\Z^3 : a+b+c=0\}.$$
A compact fundamental domain for the action is given by the hexagon
$$X=\{(x,y,z)\in\ft : |x-y|,|y-z|,|z-x|\leq 1\}$$
with vertices $\pm(\frac 23,-\frac 13,-\frac 13),\pm(-\frac 13,\frac 23,-\frac 13),\pm(-\frac 13,-\frac 13,\frac 23)$. The torus is obtained from the hexagon by identifying opposite edges.

The Weyl group of $\SU_3$ is isomorphic to $D_3$ and can be identified with the group of matrices generated by
$$\begin{pmatrix}0&1&\\1&0&\\&&-1\end{pmatrix}\text{ and }\begin{pmatrix}-1&&\\&0&1\\&1&0\end{pmatrix}.$$
The conjugation action of $W$ on $T$ corresponds to the restriction to $\ft$ of the permutation representation of $D_3\cong S_3$ on $\R^3$. The group $W_a=W_a'$ for $\SU_3$ is the semidirect product $N\rtimes W$ which acts affinely on $\ft$. (The expert reader will note that, despite the notation $N$ for the lattice, we have formally constructed $W_a'$.)

Each transposition in $W=S_3$ gives a reflection of the plane fixing a pair of vertices of the hexagon $X$. The mirror lines thus divide $X$ into equilateral triangles, any of which is a fundamental domain for the action of $N\rtimes W$ on $\ft$. This allows us to identify $N\rtimes W$ with the $(3,3,3)$-triangle group,
$$\langle s_1,s_2,s_3 | s_i^2=1, (s_is_j)^3=1, i\neq j\rangle.$$
The generators $s_1,s_2$ generate the Weyl group $W$ while $s_3$ corresponds to a reflection in the third face of an equilateral triangle and is given by the composition of a linear reflection and a translation.

The action of $N\rtimes W$ on $\ft$ is a universal example for proper actions and hence the $N\rtimes W$-equivariant $K$-homology of $\ft$ is the left-hand-side of the Baum-Connes assembly map for this group.

The right-hand-side of the Baum-Connes assembly map is the $K$-theory of $C^*_r(N\rtimes W)\cong (C^*_r N)\rtimes_r W$. By Fourier-Pontryagin duality the $C^*$-algebra $C^*_r N$ is isomorphic to the algebra of continuous functions on the torus $\widehat{N}$. We will explicitly identify this torus and the action of the Weyl group $W$ on it in this example.

The dual of $\Z^3$ is naturally identified with the 3-torus $\U^3$. Restricting a character of $\Z^3$ to $N$ yields a character of $N$ and since the inclusion of $N$ into $\Z^3$ splits we obtain all characters of $N$ in this way. The dual of $N$ is therefore a quotient of the 3-torus. Given a triple $(\alpha,\beta,\gamma)\in \U^3$ the corresponding character on $N$ is given by $(a,b,c)\mapsto \alpha^a\beta^b\gamma^c$ and since $a+b+c=0$ for elements of $N$ two such triples yield the same character precisely when they are projectively equivalent. Thus $\widehat N$ is identified with the quotient of $\U^3$ by the diagonal action of $\U$.

The tangent space of $\widehat N$ is the quotient of $\R^3$ (strictly the dual of $\R^3$) by the diagonal action of $\R$. This is canonically the dual of $\ft$, however since $\ft$ provides a transversal to the diagonal action we can identify the tangent space $\ft^*$ of $\widehat N$ with $\ft$. The character space $\widehat N$ is thus identified as the quotient of $\ft$ by some lattice which we will denote by $\Gamma$. We will now identify $\Gamma$. For a triple $(x,y,z)\in \R^3$ let $\chi_{(x,y,z)}:N\to \U$ denote the character corresponding to $(e^{2\pi ix},e^{2\pi iy},e^{2\pi iz})$. For $(a,b,c)\in N$, that is $a,b,c$ integers with $a+b+c=0$, we have
$$\chi_{(x,y,z)}(a,b,c)=e^{2\pi i(xa+yb+zc)}=e^{2\pi i((x-z)a+(y-z)b)}.$$
The lattice $\Gamma$ consists of those triples $(x,y,z)\in \R^3$ such that $\chi_{(x,y,z)}$ gives the trivial character on $N$. Thus $\Gamma$ consists of triples $(x,y,z)$ with $x+y+z=0$ and $x=y=z$ modulo $\Z$. This is precisely the kernel of the exponential map from $\ft$ to the Lie group $\PSU_3$, justifying our choice of notation $\Gamma$.

We note that $\Gamma$ is the lattice generated by the vertices of $X$:
$$\textstyle \pm(\frac 23,-\frac 13,-\frac 13),\pm(-\frac 13,\frac 23,-\frac 13),\pm(-\frac 13,-\frac 13,\frac 23).$$
A fundamental domain for the action of $\Gamma$ on $\ft$ is given by the hexagon
$$X^\vee=\{(x,y,z)\in \ft : |x|,|y|,|z|\leq \frac 13\}.$$
This has vertices $\pm(\frac 13,-\frac13,0)$, $\pm(0,\frac 13,-\frac13)$, $\pm(-\frac 13,0,\frac13)$ and again the torus is obtained by identifying opposing edges.

We will now show that the action of $W$ on $\ft$ descends to the dual action of $W$ on $\widehat N=\ft/\Gamma$. The dual action is defined by $(w\cdot \chi)((a,b,c))=\chi(w^{-1}\cdot(a,b,c))$. In particular, for $\chi=\chi_{(x,y,z)}$ we have
$$(w\cdot \chi_{(x,y,z)})((a,b,c))=\chi_{(x,y,z)}(w^{-1}\cdot(a,b,c))=e^{2\pi i\la (x,y,z),w^{-1}\cdot(a,b,c)\ra}$$
where $\la -,-\ra$ denotes the standard inner product on $\R^3$. Since the action of $W$ on $\R^3$ is isometric we have
$$e^{2\pi i\la (x,y,z),w^{-1}\cdot(a,b,c)\ra}=e^{2\pi i\la w\cdot(x,y,z),(a,b,c)\ra}=\chi_{w\cdot(x,y,z)}((a,b,c))$$
so $w\cdot\chi{(x,y,z)}=\chi_{w\cdot (x,y,z)}$ as required.

\begin{figure}[h]
\includegraphics[scale=0.3]{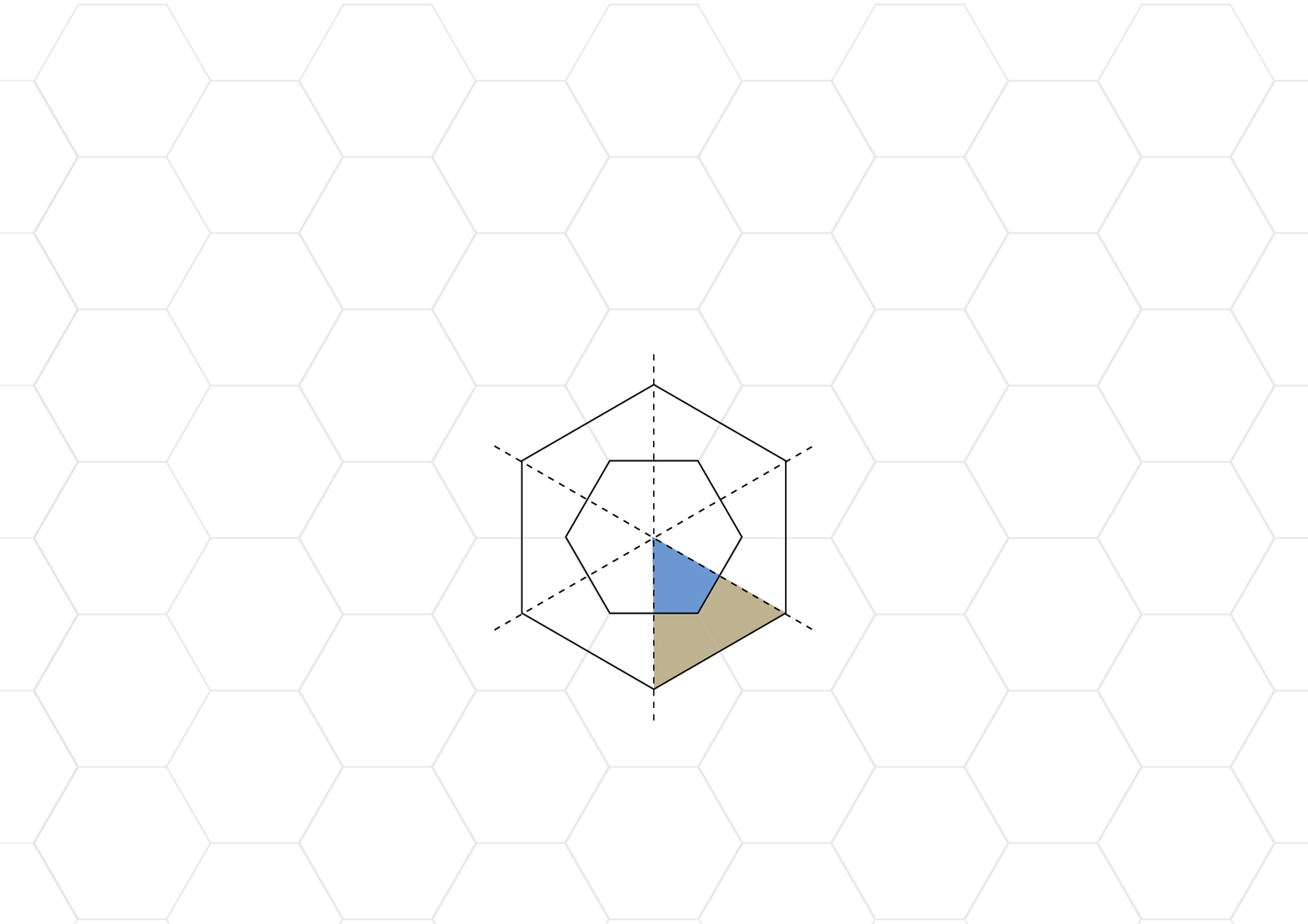}
\caption{The action of $W$ on the hexagons $X$ and $X^\vee$.}\label{fig:hexagons}
\end{figure}

We note that the reflection lines on the hexagon $X^\vee$ now bisect the edges of the hexagon, see Figure \ref{fig:hexagons}. Hence if for the moment we identify the two hexagons we have two different actions of $W$ on the hexagon corresponding to the two non-conjugate embeddings of $D_3$ in $D_6$.

\begin{figure}[h]
\includegraphics[scale=0.3]{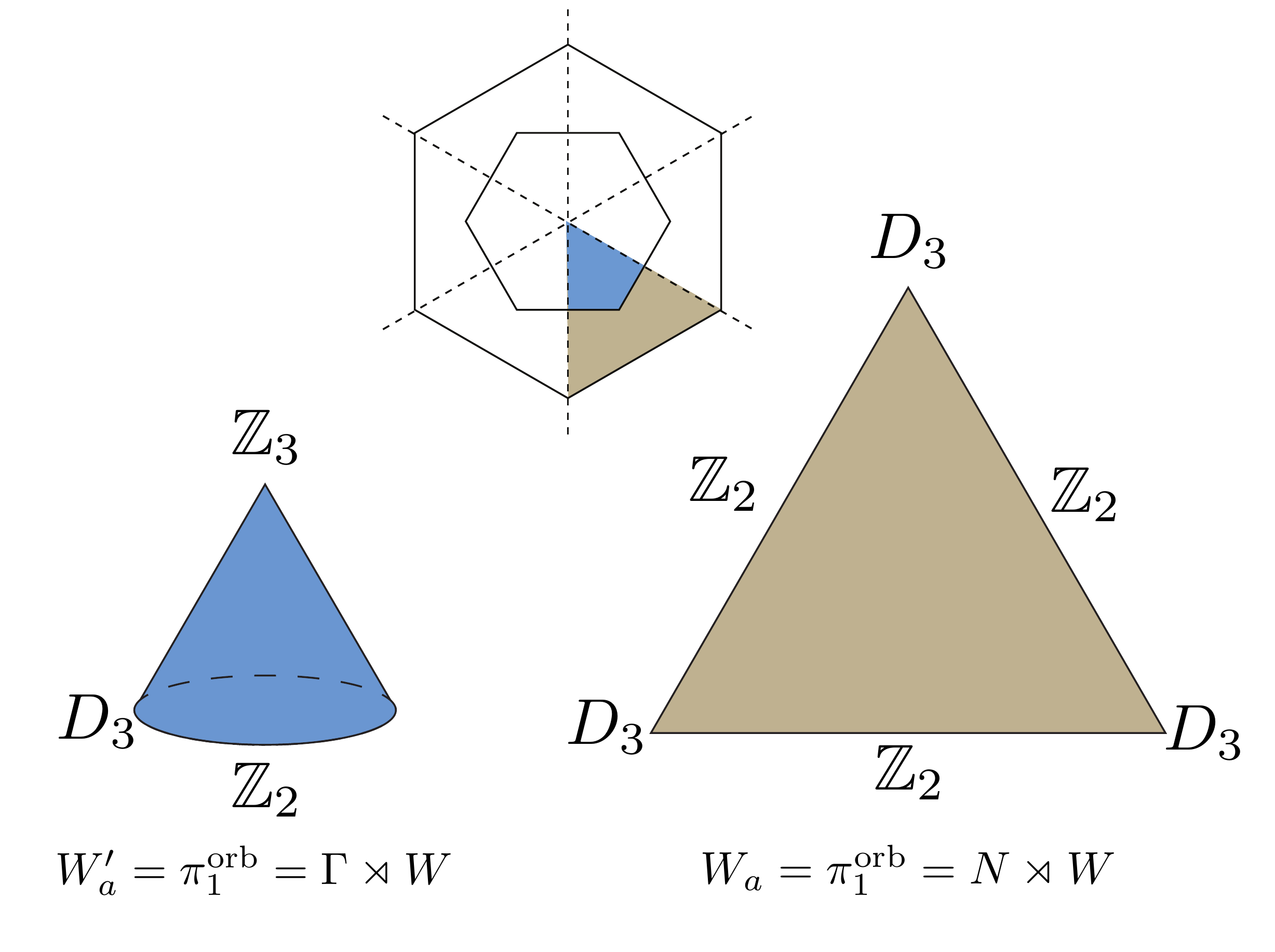}
\caption{The orbifold quotients of $\ft$ by $\Gamma\rtimes W$ and $N\rtimes W$.}\label{fig:orbifold}
\end{figure}

The orbifold quotients of the tori $T$ and $\widehat N$ by $W$ are illustrated in Figure \ref{fig:orbifold}. These may also be viewed as the orbifold quotients of $\ft$ by $N\rtimes W$ and $\Gamma\rtimes W$ respectively. Since $\Gamma$ is the kernel of the exponential map from $\ft$ to $\PSU_3$, the latter of these groups is the extended affine Weyl group of $\PSU_3$. We note that the quotient of $T$ by $W$ has three $W$-fixed points while the quotient of $\widehat N$ by $W$ has only one. On the other hand $\widehat N$ has $C_3$-isotropy at the cone point, while $T$ has no points with $C_3$-isotropy. In particular we see that there is no $W$-equivariant identification of the two tori.

The right-hand-side of the assembly map for $N\rtimes W$ is the $K$-theory of $C^*_r(N\rtimes W)$. The preceding discussion allows us to identify this as $C(\widehat N)\rtimes_r W\cong C(\ft/\Gamma)\rtimes_r W$. Hence the right-hand-side is closely related to the action of the group $\Gamma\rtimes W$ on the plane $\ft$, indeed the group $C^*$-algebra is Morita equivalent to $C_0(\ft)\rtimes_r(\Gamma\rtimes W)$. Given that the left-hand-side is determined by the action of a different group ($N\rtimes W$) on $\ft$ the appearance of $\Gamma\rtimes W$ in our description of the right-hand-side is unexpected. This illustrates and is explained by Theorem \ref{the main theorem}.

We will now proceed to compute explicitly the right-hand side of the assembly map for the triangle group. We will do so by identifying the algebra $C(\ft/\Gamma)\rtimes_r W$ as a subalgebra of the matrix algebra $M_6(C(\Delta))$ where $\Delta$ denotes the equilateral triangle with vertices
$(0,0,0),(\frac 13,\frac 13,-\frac 23), (\frac 23,-\frac 13,-\frac 13)$
(which is a fundamental domain for the action of $W$ on $X$).

We begin with the following lemma.

\begin{lemma} \label{crossed product lemma}
Let $W$ be a finite group and $A$ a $W$-$C^*$-algebra. Equip the algebra $\cB(\ell^2(W))$ with a $W$-action defined by $w\cdot S=\rho(w)S\rho(w)^*$ where $\rho$ is the right-regular representation, that is $\rho(w)\delta_x=\delta_{xw^{-1}}$. Then
$$A\rtimes_r W\simeq (A\otimes \cB(\ell^2(W)))^W$$
where $W$ acts diagonally on the tensor product.
\end{lemma}

\begin{proof}
Define $\Phi: C(\ft/\Gamma)\rtimes_r W\to C(\ft/\Gamma)\otimes \cB(\ell^2(W))$ as follows. For $a\in A$ and $w\in W$ we define
\begin{align*}
\Phi(a)&=\sum_{w\in W} w^{-1}\cdot a\otimes p_w
\\
\Phi([w])&=1\otimes \lambda(w)
\end{align*}
where $\lambda$ denotes the left-regular representation and $p_w$ denotes the rank-one projection onto $\delta_w\in\ell^2(W)$. It is straightforward to verify that this extends to a $*$-homomorphism from $A\rtimes_r W$ to $A\otimes \cB(\ell^2(W))$. The image of $\Phi$ is then contained in the $W$-invariant part of $A\otimes \cB(\ell^2(W))$ where $W$ acts diagonally on the tensor product. To see this we note that for $u\in W$ we have
$$u\cdot(\Phi(a))=\sum_{w\in W}u\cdot (w^{-1}\cdot a) \otimes \rho(u)p_w\rho(u)^*=\sum_{w\in W}(wu^{-1})^{-1}\cdot a \otimes p_{wu^{-1}}=\Phi(a)$$
while it is clear that $\Phi([w])$ is $W$-invariant.

To see that the map $\Phi$ is an isomorphism onto the $W$-invariant part of $A\otimes \cB(\ell^2(W))$ we note that its inverse can be constructed as follows. Identifying $A\otimes p_e\cB(\ell^2(W))p_e$ with $A\otimes \C=A$ we have
$$\Phi^{-1}(a\otimes S)=\sum_{u\in W} (a\otimes p_e S \lambda(u)^* p_e)[u]$$
which is clearly a left-inverse to $\Phi$ and is also a right-inverse on $W$-invariant elements of $A\otimes \cB(\ell^2(W))$.
\end{proof}

\begin{figure}[h]
\includegraphics[scale=0.3]{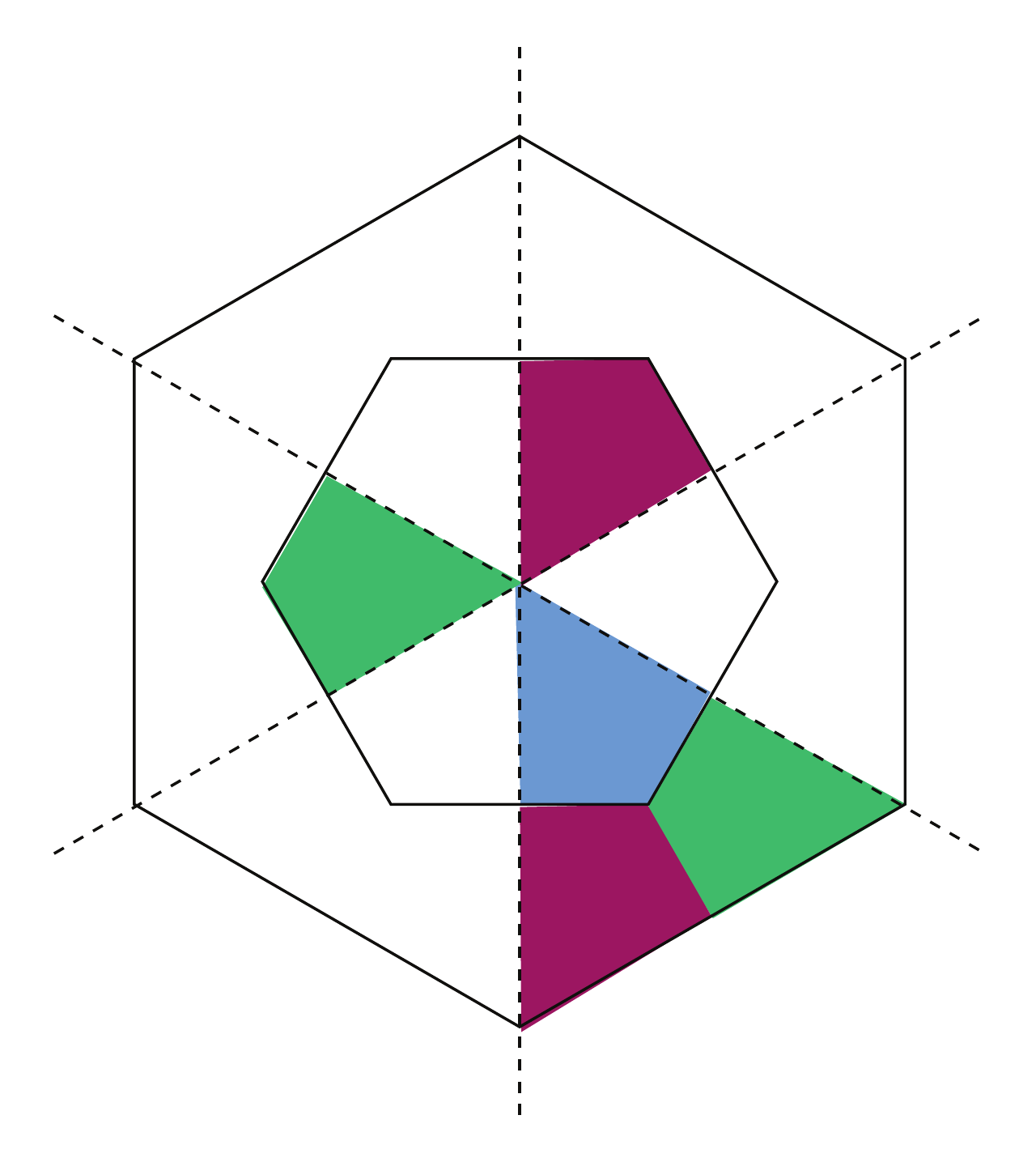}
\caption{$\Delta$ and the intersections of its translates with $X^\vee$.}\label{fig:radiation}
\end{figure}

From the lemma we have an isomorphism
$$\Phi:C(\ft/\Gamma)\rtimes_r W \to (C(\ft/\Gamma)\otimes \cB(\ell^2(W)))^W.$$
Now consider the map from $\Delta$ to the torus $\ft/\Gamma$. This induces a restriction map $\Psi$ from $(C(\ft/\Gamma)\otimes \cB(\ell^2(W)))^W$ to $C(\Delta)\otimes \cB(\ell^2(W))$ and since the action of $W$ on $\Delta$ covers $X^\vee$ (indeed it covers the whole of $X$) this restriction map is injective. We note that the image of $\Delta$ in the torus $\ft/\Gamma$ is preserved by the cyclic subgroup $C_3$ of $W$ generated by $s_1s_2$ (cf.\ Figure \ref{fig:radiation}). This action lifts to an action of $C_3$ on $\Delta$ and elements of the image of $\Psi$ are $C_3$-invariant, where $C_3$ acts diagonally on the tensor product $C(\Delta)\otimes \cB(\ell^2(W))$.

Moreover for a point on the boundary of $\Delta$ the image in $\ft/\Gamma$ is invariant under one or more of the reflections $s_1,s_2,s_1s_2s_1$ and hence the value of the function at that point is also invariant under that reflection. In particular at a vertex of $\Delta$ the value is a $W$-invariant operator on $\ell^2(W)$. This combined with the $C_3$ invariance implies automatically that the values at the three vertices agree, which indeed they must as the three vertices map to a single point in the torus.

To summarise, the composition $\Psi\circ\Phi$ gives an isomorphism from $C^*_r(N\rtimes W)$ to the subalgebra $E$ of $(C(\Delta)\otimes \cB(\ell^2(W)))^{C_3}$ of functions such that
\begin{itemize}
\item[(A)] the values at the vertices of the triangle are $W$-invariant operators on $\ell^2(W)$ and
\item[(B)] the values on an edge whose image in the torus is stabilised by $s\in \{s_1,s_2,s_1s_2s_1\}$ are $\la s\ra$-invariant operators on $\ell^2(W)$.
\end{itemize}

To compute the $K$-theory of $E$ we consider the ideal $I=(C_0(\Delta^\circ)\otimes \cB(\ell^2(W)))^{C_3}$ in $E$, where $\Delta^\circ$ denotes the interior of $\Delta$. (Note that the additional invariance conditions applied to $E$ relate to the boundary and hence are automatic for the ideal.) This fits into a short exact sequence
$$(C_0(\Delta^\circ)\otimes \cB(\ell^2(W)))^{C_3}\hookrightarrow (C(\Delta)\otimes \cB(\ell^2(W)))^{C_3}\twoheadrightarrow (C(\partial \Delta)\otimes \cB(\ell^2(W)))^{C_3}$$
where $\partial \Delta$ denote the boundary of the triangle. The $C_3$-equivariant contraction of the triangle onto its barycentre induces a homotopy equivalence from the middle term to $\cB(\ell^2(W))^{C_3}$. This in turn is Morita equivalent to $\C[C_3]$ hence its $K$-theory is $\Z^3$ in dimension $0$ and zero in dimension $1$.

\begin{remark}\label{triangle mod C_3}
The quotient term can be identified with $\cB(\ell^2(W))$-valued functions on a single edge of the triangle such that the values at the ends differ by the action of $s_1s_2\in C_3$. Specifically these operators are conjugate by $\rho(s_1s_2)$.  Since this operator is homotopic to the identity the quotient algebra is isomorphic to $C(\partial \Delta/C_3)\otimes \cB(\ell^2(W))$ and hence has the $K$-theory of a circle, i.e. $K_0=K_1=\Z$.
\end{remark}
We thus obtain the following $6$-term exact sequence.

\begin{equation}
\label{6-term for I}\begin{CD}
K_0(I)@>>> \Z^3@>>> \Z\\
@AAA @. @VVV\\
\Z @<<< 0 @<<< K_1(I)
\end{CD}
\end{equation}

The map $\Z^3\to \Z$ is the map induced on $K$-theory from the composition
\begin{align*}
\cB(\ell^2(W))^{C_3}\to& (C(\Delta)\otimes \cB(\ell^2(W)))^{C_3}\\
&\quad\to (C(\partial \Delta)\otimes \cB(\ell^2(W)))^{C_3}\to C(\partial \Delta/C_3)\otimes \cB(\ell^2(W))
\end{align*}
where the first map is inclusion as constant functions, the second is restriction and the third is the identification given in Remark \ref{triangle mod C_3}. Since the conjugation by $\rho(s_1s_2)$ fixes elements of $\cB(\ell^2(W))^{C_3}$ the overall composition is given by the inclusion of $\cB(\ell^2(W))^{C_3}$ as constant $C_3$-invariant functions in $C(\partial \Delta/C_3)\otimes \cB(\ell^2(W))$.

Taking the generators of $K_0(\C(C_3))$ to be the rank one projections corresponding to the characters of $C_3$, these each map to constant rank-$1$-projection-valued functions. Hence each generator of $K_0(\C(C_3))=\Z^3$ maps to the generator of $K_0(C(\partial \Delta/C_3)\otimes \cB(\ell^2(W)))=\Z$. We thus see that $K_1(I)$ is zero (since $\Z^3$ surjects onto $\Z$) while $K_0(I)=\Z\oplus\Z^2$. We will see that it is the cokernel $\Z^2$ of the connecting map $\Z\to K_0(I)$ that contributes to $K_0(E)$.

We now move on to the $K$-theory of the quotient $E/I$. This quotient is the subalgebra of $(C(\partial \Delta)\otimes \cB(\ell^2(W)))^{C_3}$ consisting of functions whose value at each vertex is $W$-invariant and whose value along each edge is $\la s\ra $ invariant where $s$ denotes the reflection in the edge. Again this can be identified with certain $\cB(\ell^2(W))$-valued functions on a single edge of the triangle.  Since the $C_3$-action on the triangle takes one vertex of the edge to the other, the operator appearing at the second vertex is obtain by conjugating the operator at the first vertex by $\rho(s_1s_2)$. However since the operators at each vertex are invariant under $\rho(W)$-conjugation, this conjugation is trivial and so the two operators agree.

Thus $E/I$ is identified as $\cB(\ell^2(W))^{\la s\ra}$-valued functions on $[0,1]$ whose values at the endpoints agree and are $W$-invariant operators. We therefore have a short exact sequence
$$0\to C_0(0,1)\otimes \cB(\ell^2(W))^{\la s\ra} \to E/I \to \cB(\ell^2(W))^W\to 0.$$
This short exact sequence is split, by lifting $\cB(\ell^2(W))^W$ to constant functions on $[0,1]$ hence
\begin{align*}
K_0(E/I)&=K_0(C_0(0,1)\otimes \cB(\ell^2(W))^{\la s\ra})\oplus K_0(\cB(\ell^2(W))^W)=0\oplus K_0(\C[W])=\Z^3\\
K_1(E/I)&=K_1(C_0(0,1)\otimes \cB(\ell^2(W))^{\la s\ra})\oplus K_1(\cB(\ell^2(W))^W)=K_0(\C[\la s\ra])\oplus 0=\Z^2
\end{align*}

The short exact sequence
$$0\to I \to E \to E/I\to 0$$
now yields the following $6$-term exact sequence.
$$\begin{CD}
K_0(I)=\Z^3@>>> K_0(E)@>>> K_0(E/I)=\Z^3\\
@AAA @. @VVV\\
K_1(E/I)=\Z^2 @<<< K_1(E) @<<< K_1(I)=0
\end{CD}$$
The connecting map $K_1(E/I)\to K_0(I)$ factors through the connecting map $K_1((C(\partial \Delta)\otimes \cB(\ell^2(W)))^{C_3})\to K_0(I)$ from Equation \ref{6-term for I} which we saw is an injection.  The map from $E/I$ to $C(\partial \Delta)\otimes \cB(\ell^2(W)))^{C_3}$ is the inclusion obtained by forgetting the $\la s\ra$ invariance on the interval and $W$-invariance at the vertices. The map on $K_1$ is identified with the map $K_0(\cB(\ell^2(W))^{\la s\ra })=\Z^2\to K_0(\cB(\ell^2(W)))=\Z$. This is a surjection, hence the kernel of the connecting map $K_1(E/I)\to K_0(I)$ is $\Z$ while the image agrees with the image of the connecting map from Equation \ref{6-term for I} and hence the cokernels also agree.

We conclude that
$$K_1(E)=\Z$$
while $K_0(E)$ is the direct sum of the cokernel of $K_1(E/I)\to K_0(I)$ with $K_0(E/I)=\Z^3$. The cokernel is $\Z^2$ and hence we obtain
$$K_0(E)=\Z^2\oplus \Z^3=\Z^5.$$
This completes the calculation of the $K$-theory of group $C^*$-algebra of the $(3,3,3)$-triangle group $N\rtimes W$.

We have thus established the following.

\begin{theorem}\label{triangle group}
Let $N\rtimes W$ denote the affine Weyl group of $\SU_3$  (the $(3,3,3)$-triangle group). Note that this is also the affine Weyl group of $\PSU_3$. Then
$$K_0(C^*_r(N\rtimes W))=\Z^5,\quad K_1(C^*_r(N\rtimes W))=\Z.$$
\end{theorem}

To illustrate Corollary \ref{affine theorem} we now also consider the $K$-theory of $C^*_r(\Gamma\rtimes W)$. We note that $\Gamma$ is the image of the lattice $\Z^3$ under the map
$$(a,b,c)\mapsto (a,b,c)-\frac {a+b+c}3(1,1,1).$$
The kernel is the diagonal copy of $\Z$ in $\Z^3$ and hence we can identify $\Gamma$ with $\Z^3/\Z$.

The characters on $\Gamma$ thus correspond precisely to characters on $\Z^3$ which are trivial on the diagonal copy of $\Z$. These are given by triples $(\alpha, \beta, \gamma)$ such that $\alpha\beta\gamma=1$, hence the dual $\widehat \Gamma$ of $\Gamma$ is identified with the above maximal torus $T$ for $\SU_3$. The action of $W$ on $\Gamma$ corresponds to the permutation representation on $\Z^3$ which in turn corresponds to the original action of $W$ on $T$.  We conclude that $C^*_r(\Gamma\rtimes W)\cong (C^*_r\Gamma)\rtimes_r W\cong C(T)\rtimes_r W$.

By Lemma \ref{crossed product lemma} we can now identify $C(T)\rtimes_r W$ as the $W$-invariant algebra $(C(T)\otimes \cB(\ell^2(W))^W$ where $W$ acts diagonally on the two factors. The action on the second factor is as before: $W$ acts by conjugation by the right-regular regresentation $\rho$. The triangle $\Delta$ injects into $T$ and we will now regard it as a subspace of $T$.  This gives a fundamental domain for the action of $W$ on $T$, hence the restriction map $(C(T)\otimes \cB(\ell^2(W)))^W$ to $C(\Delta)\otimes \cB(\ell^2(W))$ is injective. We denote the image by $F$. It is important to note that while the triangle $\Delta$ is the same as in the previous calculation, the triangle no longer has an action of $C_3$ on it, indeed for $w\neq e$ in $W$, the action of $w$ on $T$ takes the interior $\Delta^\circ$ of $\Delta$ entirely off itself.

The $W$-invariance condition means that if a point $x\in\Delta$ has stabiliser $H\leq W$ then the value at $x$ is an $H$-invariant operator. As noted above, the interior of $\Delta$ has trivial stabiliser. The three edges of $\Delta$ have stabilisers $\la s_1\ra,\la s_2\ra, \la s_1s_2s_1\ra$ while the vertices are stabilised by $W$. Thus we have identified $C(T)\rtimes_r W$ with the subalgebra of $C(\Delta)\otimes \cB(\ell^2(W))$ satisfying (A),(B) as above. The algebra $E$ considered previously is precisely the $C_3$-invariant subalgebra of $F$. Let $J=C_0(\Delta^\circ)\otimes \cB(\ell^2(W))$. We have $K_0(J)=\Z$ and $K_1(J)=0$. We must now determine the $K$-theory of the quotient $F/J$.

The algebra $F/J$ consists of functions on the boundary $\partial \Delta$ which are $W$-invariant at the vertices and invariant under $\la s_1\ra$, $\la s_2\ra$ and $\la s_1s_2s_1\ra$ respectively on the three edges.  Conjugating the values by $e,s_2s_1$ and $s_1s_2$ on the three edges we obtain functions which are $\la s_1\ra$ invariant on each edge.  The $W$-invariance at the vertices allows us to conjugate by different values on different edges since the vertex values are unchanged by conjugation. We consider the ideal $L$ in this algebra consisting of functions vanishing at the three vertices.  This ideal can be identified as the direct sum
$$C_0(I_1)\otimes \cB(\ell^2(W))^{\la s_1\ra}\,\,\oplus\,\, C_0(I_2)\otimes \cB(\ell^2(W))^{\la s_1\ra}\,\,\oplus\,\, C_0(I_3)\otimes \cB(\ell^2(W))^{\la s_1\ra}$$
$$\cong (C_0(I_1)\,\oplus\, C_0(I_2)\,\oplus\, C_0(I_3))\,\otimes\, \cB(\ell^2(W))^{\la s_1\ra}$$
where $I_1,I_2,I_3$ denote the interiors of the edges of the triangle. Its $K$-theory is thus $0$ in dimension $0$ and $\Z^2\oplus\Z^2\oplus \Z^2$ in dimension $1$.

The quotient by this ideal is the sum of three copies of $\C[W]$, one for each vertex. Its $K$-theory is thus $\Z^3\oplus\Z^3\oplus\Z^3$ in dimension $0$, and $0$ in dimension $1$. We obtain the following 6-term exact sequence.
$$\begin{CD}
0@>>> K_0(F/J)@>>> \Z^3\oplus\Z^3\oplus\Z^3\\
@AAA @. @VV\partial V\\
0@<<< K_1(F/J)@<<< \Z^2\oplus\Z^2\oplus \Z^2
\end{CD}$$

We remark that the algebra is contained in $C(\partial\Delta)\otimes \cB(\ell^2(W))^{\la s_1\ra}$ for which there is a corresponding $6$-term exact sequence arising from the ideal $L$. The $K$-theory of $C(\partial\Delta)\otimes \cB(\ell^2(W))^{\la s_1\ra}$ is $\Z^2$ in both dimensions and the $6$-term sequence is
$$\begin{CD}
0@>>> \Z^2@>>> \Z^2\oplus\Z^2\oplus\Z^2\\
@AAA @. @VV\partial' V\\
0@<<< \Z^2@<<< \Z^2\oplus\Z^2\oplus \Z^2
\end{CD}$$
The connecting map $\partial$ is now given by the following composition
$$\Z^3\oplus\Z^3\oplus\Z^3\xrightarrow{\iota\oplus\iota\oplus\iota} \Z^2\oplus\Z^2\oplus\Z^2 \xrightarrow{\partial'}\Z^2\oplus\Z^2\oplus\Z^2.$$
Here $\iota$ is the map on $K$-theory induced by the inclusion of $\C[W]=\cB(\ell^2(W))^{W}$ into $\cB(\ell^2(W))^{\la s_1\ra}$. The latter is identified with $M_3(\C[\la s_1\ra])$, indeed enumerating $W$ as $e,s_1,s_1s_2,s_1s_2s_1,s_2s_1,s_2$ the (right) $s_1$-invariant operators are matrices of the form
$$\left(\begin{array}{c|c|c}
M_{11}&M_{12}&M_{13}\\
\hline
M_{21}&M_{22}&M_{23}\\
\hline
M_{31}&M_{32}&M_{33}\\
\end{array}\right)$$
with each $M_{ij}$ a $2\times 2$-matrix in $\C[\la s_1\ra]$. This allows us to explicitly identify the map $\iota$ as follows. The generators of $K_0(\C[W])$ are given by three projections $p_t,p_s,p_d$, where $p_t$ maps to a rank one projection under the trivial representation and vanishes under the sign and dihedral representations, and similarly for $p_s,p_d$. Explicitly
$$p_t=\frac 16\sum_{w\in W} [w], \quad p_s=\frac 16\sum_{w\in W}\mathrm{sign}(w)[w],\quad p_d=\frac 12([e]+[s_1])-p_t.$$
Letting $q_t,q_s$ denote the projections in $\C[\la s_1\ra]$ corresponding to the trivial and sign representations of $\la s_1\ra$ we have
$$p_t=\frac 13\left(\begin{array}{c|c|c}
q_t&q_t&q_t\\
\hline
q_t&q_t&q_t\\
\hline
q_t&q_t&q_t\\
\end{array}\right),\quad p_s=\frac 13\left(\begin{array}{c|c|c}
q_s&q_s&q_s\\
\hline
q_s&q_s&q_s\\
\hline
q_s&q_s&q_s\\
\end{array}\right),$$
$$
p_d+p_t=\frac 12([e]+[s_1])=\frac 12\left(\begin{array}{c|c|c}
2q_t&&\\
\hline
&q_t+q_s&q_t-q_s\\
\hline
&q_t-q_s&q_t+q_s\\
\end{array}\right).
$$
Hence at the level of $K$-theory the map is given by $[p_t]\mapsto [q_t],[p_s]\mapsto [q_s]$ and $[p_d]+[p_t]\mapsto 2[q_t]+[q_s]$. Hence $[p_d]\mapsto [q_t]+[q_s]$.

Since $\iota$ is surjective, the image, and hence also the cokernel, of $\partial$ is the same as for $\partial'$.  Hence we see that $K_1(F/J)=\Z^2$. The kernel of $\partial'$ is the image of $\Z^2$ in $\Z^2\oplus\Z^2\oplus\Z^2$ under the diagonal inclusion. The kernel of $\partial$ is thus the preimage of this diagonal $\Z^2$ under $\iota\oplus\iota\oplus\iota$. Since $[p_t]+[p_s]-[p_d]$ is in the kernel of $\iota$, this element of $K_0(\C[W])$ gives three elements of the kernel of $\partial$, one at each vertex. We have two additional generators corresponding to the elements of the diagonal $\Z^2$. Explicitly we can take these to be $[p_t \oplus p_t \oplus p_t]$ and $[p_s \oplus p_s \oplus p_s]$. Hence we conclude that $K_0(F/J)=\Z^5$.

To complete the calculation of the $K$-theory of $F$ we now consider the $6$-term exact sequence arising from the short exact sequence
$$0\to J\to F\to F/J\to 0$$
This gives the following.
$$
\begin{CD}
K_0(J)=\Z @>>> K_0(F) @>>>K_0(F/J)=\Z^5\\
@AAA@.@VVV\\
K_1(F/J)=\Z^2@<<<K_1(F) @<<< K_1(J)=0
\end{CD}
$$
It remains to determine the connecting map from $K_1(F/J)$ to $K_0(J)$. We saw that the forgetful map $F/J\to C(\partial \Delta)\otimes \cB(\ell^2(W))^{\la s_1\ra}$ induced an isomorphism on $K$-theory. We observe that the further inclusion into $C(\partial \Delta)\otimes \cB(\ell^2(W))$ gives the summation map $\Z^2\to \Z$ on $K_1$ (and also on $K_0$). The connecting map $K_1(F/J)\to K_0(J)$ factors through the connecting map for the exact sequence
$$0\to J \to C(\Delta)\otimes \cB(\ell^2(W)) \to C(\partial \Delta)\otimes \cB(\ell^2(W))\to 0$$which is an isomorphism $K_1(C(\partial \Delta)\otimes \cB(\ell^2(W)))=\Z\cong K_0(J)=\Z$. Thus the map from $K_1(F/J)$ to $K_0(J)$ is again the summation map from $\Z^2$ to $\Z$.

We thus conclude that $K_1(F)\cong \Z$ and $K_0(F)\cong K_0(F/J)\cong \Z^5$ establishing the following theorem.

\begin{theorem}\label{cone group}
Let $\Gamma\rtimes W$ denote the extended affine Weyl group of $\PSU_3$. Then
$$K_0(C^*_r(\Gamma\rtimes W))=\Z^5,\quad K_1(C^*_r(\Gamma\rtimes W))=\Z.$$
\end{theorem}

\begin{remark}
The observation that the $K$-theory groups appearing in Theorems \ref{triangle group} and \ref{cone group} are the same is an illustration of Corollary \ref{affine theorem} for the Lie group $\PSU_3$. Indeed in this example we see that the affine Weyl group and extended affine Weyl group have $K$-theory which is integrally isomorphic, not just rationally isomorphic.
\end{remark}

\section{Langlands duality}   \label{langlands section}

As discussed in the introduction one of the key motivations of this paper is that for extended affine Weyl groups, the Baum-Connes correspondence should be thought of as an equivariant duality between the tori $T$ and $\widehat{\Gamma}$. In this section we will recall the definition of Langlands duality for complex Lie groups and how this can be used to construct Langlands duality in the real case.  We will show that the aforementioned duality of the tori corresponds to this real Langlands duality for the Lie group.

\subsection{Complex reductive groups} Let $\bH$ be a connected complex reductive group, with maximal torus $\bS$.   This determines a root datum
\[
R(\bH,\bS): = (\bX^*(\bS), R, \bX_*(\bS), R^\vee)
\]

Here $R$ and $R^\vee$ are the sets of roots and coroots of $\bH$, while 
\begin{align}
\bX^*(\bS): = \Hom(\bS, \C^\times) \quad \textbf{and} \quad  \bX_*(\bS) : = \Hom(\C^\times,\bS)
\end{align}
 are its character and co-character lattices.

The root datum implicitly includes the pairing $\bX^*(\bS) \times \bX_*(\bS)  \to \Z$ and the bijection $R \to R^\vee, \; \alpha \mapsto h_\alpha$ between roots and coroots. Root data classify complex reductive Lie groups, in the sense that two such groups are isomorphic if and only if their root data are isomorphic (in the obvious sense).

Interchanging the roles of roots and coroots and of the character and co-character lattices results in a new root datum:
\[
R(\bH, \bS)^\vee : = (\bX_*(\bS), R^\vee, \bX^*(\bS), R)
\]

The Langlands dual group of $\bH$ is the complex reductive group $\bH^\vee$ (unique up to isomorphism) determined by the dual root datum 
$R(\bH, \bS )^\vee$.
A root datum also implies a choice of maximal torus $\bS \subset \bH$ via the canonical isomorphism $\bS \simeq \Hom(\bX^*(\bS), \C^\times)$,
and likewise a natural choice of maximal torus for the Langlands dual group $\bH^\vee:  \bS^\vee := \Hom(\bX_*(\bS), \C^\times) \subset \bH^\vee$.

In particular, we have the equation
\begin{align}\label{dualtorus}
\bX^*(\bS^\vee) = \bX_*(\bS)
\end{align}

\subsection{Compact semisimple groups}\label{compact semisimple groups}
 Let now $G$ be a compact connected semisimple Lie group, with maximal torus $T$.   We recall that a compact connected Lie group is semisimple 
if and only if it has finite centre \cite[p.285]{B}.   The classical examples are the compact real forms
 \[
 \SU_n, \,\SO_{2n + 1},\, \Sp_{2n},\, \SO_{2n},\, E_6,\, E_7,\, E_8,\, F_4,\, G_2.
 \]

For a Lie group $G$, the \emph{complexification} $G_\C$ is a complex Lie group together with a morphism from $G$, satisfying the universal property that for any morphism of $G$ into a complex Lie group $\bL$ there is a unique factorisation through $G_\C$.

The complexification $\bS: = T_\C$ of $T$ is a maximal torus in $\bH: = G_\C$, and so the dual torus $\bS^\vee$ is well-defined in 
the dual group $\bH^\vee$.   Then $T^\vee$ is defined to be the maximal compact subgroup of $\bS^\vee$, and satisfies the condition
\[
(T^\vee)_\C =  \bS^\vee.
\]
By definition, the torus $T^\vee$ is the $T$-dual of $T$.

We denote by $X^*(T)$ the group of morphisms from the Lie group $T$ to the Lie group $\U = \{z \in \C : |z| = 1\}$, and denote by $X_*(T)$ the group of morphisms from $\U$ to $T$.  Corresponding to (\ref{dualtorus}), we have the $T$-duality equation
\begin{align}\label{dualtorus2}
X^*(T^\vee) = X_*(T)
\end{align}
which follows from the identification of the lattices $X^*(T^\vee), X_*(T)$ with their complex counterparts $\bX^*(\bS^\vee),$$\bX_*(\bS)$.  Moreover since $\bS^\vee=\Hom(\bX_*(\bS), \C^\times)$  the maximal torus is $T^\vee=\Hom(\bX_*(\bS), \U)=\Hom(X_*(T), \U)$.

\begin{definition}
 The \emph{Langlands dual} of $G$, denoted $G^\vee$, is defined to be a maximal compact subgroup of $\bH^\vee$ containing the torus $T^\vee$.
\end{definition}

The process of passing to a maximal compact subgroup is inverse to complexification in the sense that complexifying $G^\vee$ recovers $\bH^\vee$.

\medskip
\subsubsection{A table of Langlands dual groups}  Given a compact connected semisimple Lie group $G$, the product $|\pi_1(G)| \cdot |\cZ(G)|$ is unchanged by Langlands duality, i.e.\ it agrees with the product $|\pi_1(G^\vee)| \cdot |\cZ(G^\vee)|$. This product is equal to the \emph{connection index}, denoted $f$, (see \cite[IX, p.320]{B}), which is defined in \cite[VI, p.240]{B}. The connection indices are listed in \cite[VI, Plates I--X,
p.265--292]{B}.

The following is a table of Langlands duals and connection indices for compact connected semisimple groups: 

\begin{center}
\begin{tabular}{l||l|l}
$G$ & $G^{\vee}$ & $f$\\
\hline
$A_n = \SU_{n + 1}$ & $\PSU_{n + 1}$  & $n+1$\\
$B_n = \SO_{2n+1}$ & $\Sp_{2n}$ & $2$\\
$C_n = \Sp_{2n}$ & $\SO_{2n+1}$ & $2$\\
$D_n = \SO_{2n}$ & $\SO_{2n}$ &$4$\\
$E_6$ & $E_6$ & $3$\\
$E_7$ & $E_7$ & $2$\\
$E_8$ & $E_8$ &$1$\\
$F_4$ & $F_4$ & $1$\\
$G_2$ & $G_2$ & $1$\\
\end{tabular}
\end{center}

In this table, 
the simply-connected form of $E_6$ (resp. $E_7$) dualises to the adjoint form of $E_6$ (resp. $E_7$).

The Lie group $G$ and its dual $G^\vee$ admit a common Weyl group 
\[
W = N(T)/T = N(T^\vee)/T^\vee.
\]
The $T$-duality equation (\ref{dualtorus2}) identifies the action of the Weyl group of $T$ on $X_*(T)$ with the dual action of the Weyl group of $T^\vee$ on $X^*(T^\vee)$.

Note that, in general,  $T$ and $T^\vee$ are \emph{not} isomorphic as $W$-spaces.  For example, let $G = \SU_3$.  Then $G^\vee = \PSU_3$ and we have
\[
T = \{(z_1, z_2, z_3) : z_j \in \U, z_1 z_2 z_3 = 1\}
\]
as in the Section \ref{triangle section} and
\[ 
T^\vee = \{(z_1: z_2: z_3) : z_j \in \U, z_1 z_2 z_3 = 1 \}
\]
the latter being in homogeneous coordinates.  We remark that $T^\vee$ can be identified with the torus $\widehat N$ from the previous section, indeed $T^\vee$ is the Pontryagin dual of $X_*(T)$ which in the case of $\SU_3$ is the group $N$ since $\SU_3$ is simply connected. The map 
\[
T \to T^\vee, \quad (z_1, z_2, z_3) \mapsto (z_1 : z_2 : z_3)
\]
 is a $3$-fold cover: the
pre-image of $(z_1 : z_2 : z_3)$ is the set \[
\{(\eta z_1, \eta z_2, \eta z_3) : \eta \in \U, \eta^3 = 1\}.
\]

The torus $T$ admits three $W$-fixed points,
namely
\[
\{(1,1,1), (\omega, \omega, \omega), (\omega^2, \omega^2, \omega^2) : \omega = \exp(2 \pi i/3)\}
\]
whereas the unique $W$-fixed point in $T^\vee$ is the identity $(1:1:1) \in T^\vee$, hence $T$ and $T^\vee$ are not equivariantly isomorphic.

 \subsection{The nodal group}   The \emph{nodal group of $T$} is defined to be
\[
\Gamma(T): = \ker (\exp: \ft \to T)
\]
where $\ft$ denotes the Lie algebra of $T$. Differentiating the action of $W$ on $T$ via automorphisms we obtain a linear action of $W$ on $\ft$. The map $\exp$ is $W$-equivariant and hence the action of $W$ restricts to an action on $\Gamma(T)$.

\begin{lemma}\label{gamma}
There is a $W$-equivariant isomorphism
\[
X_*(T) \simeq \Gamma(T)
\]
\end{lemma}

\begin{proof}   The group $X_*(T)$ is the group  of morphisms from the Lie group $\U$ to the Lie group $T$.
Given $f \in X_*(T)$, we have a commutative diagram
$$\begin{CD}
\Gamma(\U) @>>> \fu=i\R @>\exp>>\U\\
@VV{\Gamma(f)}V @VV{df}V @VV{f}V\\
\Gamma(T) @>>> \ft @>\exp>>T
\end{CD}$$
where $\Gamma(f)$ is given by restricting $df$ to $\Gamma(\U)$.
   We identify $\Gamma(\U)$ with the subgroup $2\pi i \Z$ of $\fu = i\R$.   The homomorphism $\Gamma(f)$ is determined by its value on the generator $2\pi i$ and we define a homomorphism
   \[
   X_*(T) \to \Gamma(T), \quad \quad f \mapsto \Gamma(f)(2\pi i).
   \]
Conversely, given $\gamma\in\Gamma(T)$ there is a unique linear map from $\fu = i\R$ to $\ft$ taking $2\pi i\mapsto \gamma$. Composing with the exponential map we obtain a map from $\fu$ to $T$, and since $\gamma$ is in the kernel of the exponential map this descends to a morphism from $\U$ to $T$. This gives a homomorphism from $\Gamma(T)$ to $X_*(T)$ which inverts the above homomorphism as in \cite[p.307]{B}.

The isomorphism is equivariant since for $f\in X_*(T)$ we have
$$\Gamma(w\cdot f)(2\pi i)=d(w\cdot f)(2\pi i)=(w\cdot df)(2\pi i)=w\cdot (\Gamma(f)(2\pi i)).$$
\end{proof}

We now observe that the groups $\Gamma(T)$ and $T^{\vee}$ are \emph{in duality} in the sense of locally compact abelian topological groups.

\begin{lemma}\label{gammadual}  Let $\widehat{\Gamma}$ denote the Pontryagin dual of $\Gamma=\Gamma(T)$.   Then we have a $W$-equivariant isomorphism 
\[
\widehat{\Gamma} \simeq T^\vee.
\]
and hence an isomorphism of $W$\!-\,$C^*$\!-algebras
\[
C^*_r(\Gamma) \simeq C(T^\vee).
\]
\end{lemma}

\begin{proof} By Lemma \ref{gamma} the nodal group $\Gamma(T)$ is $W$-equivariantly isomorphic to $X_*(T)$. By the $T$-duality equation (\ref{dualtorus2}) the latter is equal to
$$X^*(T^\vee)=\Hom(T^\vee,\U).$$
This by definition is the Pontryagin dual $\widehat{T^\vee}$.
The result now follows by Pontryagin duality.  \end{proof}

\section{Equivariant Poincar\'e Duality between $C(T)$ and $C(T^\vee)$}\label{poincare section}

Let $G$ be a compact connected semisimple Lie group with maximal torus $T$.  Let $T^\vee$ be the dual torus as in Section \ref{compact semisimple groups} and let $W$ denote the Weyl group of $G$. In this section we will establish the isomorphism from $KK^*_W(C(T),\C)$ to $KK^*_W(\C,C(T^\vee))$.  We will do this by exhibiting a $W$-equivariant Poincar\'e duality between the algebras $C(T)$ and $C(T^\vee)$.

We remark that the standard Poincar\'e duality of Kasparov, \cite{K},  provides an equivariant duality from $C(T)$ to $C(T,\cl(\ft^*))$, however the introduction of the Clifford algebra gives a `dimension shift' which does not appear in the assembly map. We say `dimension shift' in quotes since the appearance of the Clifford algebra would give a dimension shift if it carried a trivial action, but is more subtle in the case when the group also acts on the Clifford algebra.

We recall that for $\fG$-$C^*$-algebras $A,B$ a Poincar\'e duality is given by elements $\fa\in KK_\fG(B\tensor A,\C)$\footnote{It is conventional to take $\fa\in KK_\fG(A\tensor B,\C)$ however we have selected this alternative notational convention to favour the computation of $\fb\otimes_B\fa$.} and $\fb\in KK_\fG(\C,A\tensor B)$ with the property that
\begin{align*}
\fb\otimes_A\fa &= 1_B \in KK_\fG(B,B)\\
\fb\otimes_B\fa &= 1_A \in KK_\fG(A,A).
\end{align*}
These then yield isomorphisms between the $K$-homology of $A$ and the $K$-theory of $B$ (and vice versa) given by
\begin{align*}
\fx \mapsto \fb\otimes_A\fx \in KK_\fG(\C,B) \text{ for } \fx \in KK(A,\C)\\
\fy \mapsto \fy\otimes_B\fa \in KK_\fG(A,\C) \text{ for } \fy \in KK(\C,B).
\end{align*}
We will therefore construct elements in the groups $KK_W(C(T^\vee)\tensor C(T),\C)$ and $KK_W(\C,C(T)\tensor C(T^\vee))$.  We begin with the latter.

\subsection{Construction of the element in $KK_W(\C,C(T)\tensor C(T^\vee))$}

Let $C_c(\ft)$ denote the space of continuous compactly supported functions on $\ft$ and equip this with a $C(T)\otimes C(T^\vee)$-valued inner product defined by
$$\la \phi_1,\phi_2\ra(x,\eta)=\sum_{\alpha,\beta \in\Gamma} \overline{\phi_1(x-\alpha)}\phi_2(x-\beta)e^{2\pi i \la \eta,\beta-\alpha\ra}.$$
We remark that the support condition ensures that this is a finite sum, and that it is easy to check that $\la \phi_1,\phi_2\ra(x,\eta)$ is $\Gamma$-periodic in $x$ and $\Gamma^\vee$-periodic in $\eta$.

The space $C_c(\ft)$ has a $C(T)\otimes \C[\Gamma]$-module structure
$$(\phi\cdot (f\otimes [\gamma])) = \phi(x+\gamma) \tilde f(x)$$
where we view the function $f$ in $C(T)$ as a $\Gamma$-periodic function $\tilde{f}$ on $\ft$. We have
\begin{align*}
\la \phi_1,\phi_2\cdot (f\otimes [\gamma])\ra(x,\eta)&=\sum_{\alpha,\beta \in\Gamma} \overline{\phi_1(x-\alpha)}\phi_2(x-\beta+\gamma)\tilde f(x-\beta)e^{2\pi i \la \eta,\beta-\alpha\ra}.\\
&=\sum_{\alpha,\beta' \in\Gamma} \overline{\phi_1(x-\alpha)}\phi_2(x-\beta')e^{2\pi i \la \eta,\beta'-\alpha\ra}\tilde f(x)e^{2\pi i \la \eta,\gamma\ra}.\\
&=\la \phi_1,\phi_2\ra(x,\eta)\tilde f(x)e^{2\pi i \la \eta,\gamma\ra}.
\end{align*}
Now completing $C_c(\ft)$ with respect to the norm arising from the inner product, the module structure extends by continuity to give $\ol{C_c(\ft)}$ the structure of a $C(T)\tensor C^*_r(\Gamma)\cong C(T)\tensor C(T^\vee)$ Hilbert module. We denote this Hilbert module by $\cE$ and give this the trivial grading.

The group $W$ acts on $\ft$ and hence on $C_c(\ft)$ by $(w\cdot \phi)(x)=\phi(w^{-1}x)$. We have
$$(w\cdot(\phi\cdot (f\otimes [\gamma])))(x) = \phi(w^{-1}x+\gamma) \tilde f(w^{-1}x)=((w\cdot \phi)\cdot(w\cdot f\otimes[w\gamma]))(x)$$
so the action is compatible with the module structure.  Now for the inner product we have

\begin{align*}
\la w\cdot\phi_1,w\cdot\phi_2\ra(x,\eta)&=\sum_{\alpha,\beta \in\Gamma} \overline{(w\cdot\phi_1)(x-\alpha)}(w\cdot\phi_2)(x-\beta)e^{2\pi i \la \eta,\beta-\alpha\ra}\\
&=\sum_{\alpha,\beta \in\Gamma} \overline{\phi_1(w^{-1}x-w^{-1}\alpha)}\phi_2(w^{-1}x-w^{-1}\beta)e^{2\pi i \la \eta,\beta-\alpha\ra}\\
&=\sum_{\alpha',\beta' \in\Gamma} \overline{\phi_1(w^{-1}x-\alpha')}\phi_2(w^{-1}x-\beta')e^{2\pi i \la \eta,w(\beta'-\alpha')\ra}\\
&=\sum_{\alpha',\beta' \in\Gamma} \overline{\phi_1(w^{-1}x-\alpha')}\phi_2(w^{-1}x-\beta')e^{2\pi i \la w^{-1}\eta,\beta'-\alpha'\ra}\\
&=(w\cdot\la \phi_1,\phi_2\ra)(x,\eta).
\end{align*}
Hence $\cE$ is a $W$-equivariant Hilbert module.

We observe that the identity on $\cE$ is a compact operator. To see this we note that if $\psi$ is supported inside a single fundamental domain for the action of $\Gamma$ on $\ft$ then
\begin{align*}
(\psi\la\psi,\phi\ra)(x)&=(\psi\cdot\sum_{\alpha,\beta \in\Gamma}\ol{\psi(x-\alpha)}\phi(x-\beta)[\beta-\alpha])(x)\\
&=\sum_{\alpha,\beta \in\Gamma}\psi(x-\beta-\alpha)\ol{\phi(x-\alpha)}\phi(x-\beta)\\
&=\sum_{\alpha \in\Gamma}|\psi(x-\alpha)|^2\phi(x).
\end{align*}
Hence the rank one operator $\phi\mapsto \psi\la\psi,\phi\ra$ is multiplication by $|\psi(x)|^2$ extended $\Gamma$-periodically from the fundamental domain to $\ft$. By taking a $\Gamma$-equivariant partition of unity we can therefore exhibit the multiplication by $1$ as a sum of rank one operators.

We thus have the following result.

\begin{theorem}
The triple $(\cE,1,0)$, where $1$ denote the identity representation of $\C$ on $\cE$, is a $W$-equivariant Kasparov triple defining an element $\fb$ in $KK_W(\C,C(T)\tensor C(T^\vee))$.
\end{theorem}

\subsection{Construction of the element in  $KK_W(C(T^\vee)\tensor C(T),\C)$}\label{K-homology element}

Let $\{\e_j=\frac{\partial}{\partial x^j}\}$ be an orthonormal basis for $\ft$ and let $\{\eps^j\}$ denotes the dual basis of $\ft^*$. We we usually consider these as generators of the Clifford algebra $\cl(\ft\times\ft^*)$.

Consider the projection $P=\prod_{j}\frac 12(1-i\e_j\eps^j)$. We will show (Proposition \ref{invariance of symmetric elements}) that this is $W$-invariant with respect to the diagonal action of $W$ on $\ft\times \ft^*$ (the action of $W$ on $\ft^*$ is the dual action induced by the action on $\ft$). The corner algebra $P\cl(\ft\times \ft^*)P$ is $\C P$, and we will identify this with $\C$. Now take $\cS$ to be the space of spinors
$$\cS=\cl(\ft\times \ft^*)P$$
which is a finite dimensional Hilbert space, with inner product given by $\la aP,bP\ra=Pa^*bP$. This is naturally equipped with a representation of $\cl(\ft\times \ft^*)$ by left multiplication and invariance of $P$ with respect to $W$ implies that the action of $W$ on the Clifford algebra restricts to an action on $\cS$.

Our $K$-homology element will be represented by an unbounded Kasparov triple with Hilbert space $L^2(\ft)\tensor \cS$.

We must now construct a representation of $C(T^\vee)\tensor C(T)$ on $L^2(\ft)\tensor \cS$.  To build the representation it suffices to define commuting representations of $C(T^\vee)\tensor 1$ and $1\tensor C(T)$. The representation of $C(T)$ is the usual pointwise multiplication on $L^2(\ft)$ viewing elements of $C(T)$ as $\Gamma$-periodic functions on $\ft$. The representation of $C(T^\vee)$ involves the action of $\Gamma$ on $\ft$. We introduce the following notation which is more general than we require at this point.

For $a$ an affine isometry of $\ft$, let $L_a$ be the operator on $L^2(\ft)$ induced by the action of $a$ on $\ft$:
$$(L_a\xi)(y)=\xi(a^{-1}\cdot y).$$
Here we will consider $L_\gamma$ for $\gamma\in \Gamma$ (acting by translations on $\ft$), however we note that we can also use general elements of $W_a'=\Gamma\rtimes W$, and translations by arbitrary elements of $\ft$.

Now for the function $\eta\mapsto e^{2\pi i\la \eta,\gamma\ra}$ in $C(T^\vee)$ we define
$$\rho(e^{2\pi i\la \eta,\gamma\ra})=L_\gamma\otimes 1_\cS.$$
Identifying $C(T^\vee)$ with $C^*_r(\Gamma)$ and identifying $L^2(\ft)$ with $\ell^2(\Gamma)\otimes L^2(X)$ where $X$ is a fundamental domain for the action of $\Gamma$, the representation of the algebra is given by the left regular representation on $\ell^2(\Gamma)$.

We define an unbounded operator $Q_0:L^2(\ft)\tensor \cS\to L^2(\ft)\tensor \cS$ (using Einstein summation convention) by
$$Q_0(v\otimes a)=\frac{\partial v}{\partial y^j}\otimes \eps^j a-2\pi i y^jv(y)\otimes \e_ja.$$
This operator combines the Dirac operator on $\ft$ with the inverse Fourier transform of the Dirac operator on $\ft^*$.  It combines the creation and annihilation operators for the Quantum Harmonic Oscillator, see Theorem \ref{ladder} below.

Consider the commutators of $Q_0$ with the representation $\rho$.  For $f\in C(T)$, the operator $\rho(f)$ commutes exactly with the second term $2\pi i y^j\otimes \e_j$ in $Q_0$, while, for $f$ smooth, the commutator of $\rho(f)$ with $\frac{\partial }{\partial y^j}\otimes \eps^j$ is given by the bounded operator $\frac{\partial f}{\partial y^j}\otimes \eps^j$. Now for the function $\eta\mapsto e^{2\pi i\la \eta,\gamma\ra}$ in $C(T^\vee)$ we have $\rho(e^{2\pi i\la \eta,\gamma\ra})=L_\gamma\otimes 1_\cS$.  This commutes exactly with the differential term of the operator, while
$$L_\gamma(2\pi i y^j)L_\gamma^*=2\pi i (y^j-\gamma^j)$$
hence the commutator $[L_\gamma\otimes 1_\cS,2\pi i y^j\otimes \e_j]$ is again bounded.

We have verified that $Q_0$ commutes with the representation $\rho$ modulo bounded operators, on a dense subalgebra of $C(T^\vee)\tensor C(T)$.  Thus to show that the triple
$$(L^2(\ft)\tensor \cS,\rho,Q_0)$$
is an unbounded Kasparov triple it remains to prove the following.

\begin{theorem}\label{ladder}
The operator $Q_0$ has compact resolvent.  It has $1$-dimensional kernel with even grading.
\end{theorem}

\begin{proof}
In the following argument we will \emph{not} use summation convention. We consider the following operators on $L^2(\ft)\tensor \cS$:
\begin{align*}p_j&=\frac{\partial}{\partial y^j}\otimes \eps^j\\
x_j&=-2\pi iy^j \otimes \e_j\\
q_j&=\frac{1}{2}(1+1\otimes i\e_j\eps^j)\\
A_j&=\frac{1}{2\sqrt{\pi}}(p_j+x_j)
\end{align*}
Since $A_j$ anti-commutes with $1\otimes i\e_j\eps^j$ we have $q_jA_j=A_j(1-q_j)$, hence we can think of $A_j$ as an off-diagonal matrix with respect to $q_j$. We write $A_j$ as $a_j+a_j^*$ where $a_j=q_jA_j=A_j(1-q_j)$ and hence $a_j^*=A_jq_j=(1-q_j)A_j$. We think of $a_j^*$ and $a_j$ as creation and annihilation operators respectively and we define a number operator $N_j=a_j^*a_j$. The involution $i\eps^j$ intertwines $q_j$ with $1-q_j$. We define $A_j',N_j'$ to be the conjugates of $A_j,N_j$ respectively by $i\eps^j$.  Note that
$$A_j'=\frac{1}{2\sqrt{\pi}}(p_j-x_j)$$
and hence
$$A_j^2=(A_j')^2+2\frac{1}{4\pi}[x_j,p_j]=(A_j')^2+1\otimes i\e_j\eps^j.$$
We have $N_j'=A_j'(1-q_j)A_j'=q_j(A_j')^2$. Thus
$$a_ja_j^*=q_jA_j^2q_j=q_jA_j^2=q_j(A_j')^2+q_j(1\otimes i\e_j\eps^j)=N_j'+q_j.$$
Hence the spectrum of $a_ja_j^*$ (viewed as an operator on the range of $q_j$) is the spectrum of $N_j'$ shifted by $1$. However $N_j'$ is conjugate to $N_j=a_j^*a_j$ so we conclude that
$$\Sp (a_ja_j^*)=\Sp (a_j^*a_j)+1.$$
But $\Sp (a_ja_j^*)\setminus\{0\}=\Sp (a_j^*a_j)\setminus\{0\}$ so we conclude that $\Sp (a_j^*a_j)=\{0,1,2,\dots\}$ while $\Sp (a_ja_j^*)=\{1,2,\dots\}$.

Now since the operators $A_j$ pairwise gradedly commute we have
$$Q_0^2=4\pi\sum_{j} A_j^2=4\pi\sum_{j} a_j^*a_j+a_ja_j^*$$
and noting that the summands commute we see that $Q_0^2$ has discrete spectrum.  To show that $(1+Q_0^2)^{-1}$ is compact, it remains to verify that $\ker Q_0$ is finite dimensional (and hence that all eigenspaces are finite dimensional). We have
$$\ker Q_0 = \ker Q_0^2 = \bigcap_j\ker  A_j^2=\bigcap_j \ker A_j.$$

Multiplying the differential equation $(p_j+x_j)f=0$ by $-\exp(\pi (y^j)^2\otimes i\eps^j\e_j)\eps^j$ we see that the kernel of $A_j$ is the space of solutions of the differential equation
$$\frac{\partial }{\partial y^j}(\exp(\pi (y^j)^2\otimes i\eps^j\e_j)f)=0$$whence for $f$ in the kernel we have
$$f(y^1,\dots,y^n)=\exp(-\pi (y^j)^2\otimes i\eps^j\e_j)f(y^1,\dots,y^{j-1},0,y^{j+1},\dots,y^n).$$
Since the solutions must be square integrable the values of $f$ must lie in the $+1$ eigenspace of the involution $i\eps^j\e_j$, that is, the range of the projection $1-q_j$. On this subspace the operator $\exp(-\pi (y^j)^2\otimes i\eps^j\e_j)$ reduces to $e^{-\pi (y^j)^2}(1-q_j)$.  Since the kernel of $Q_0$ is the intersection of the kernels of the operators $A_j$ an element of the kernel must have the form
$$f(y)=e^{-\pi|y|^2}\prod_j(1-q_j)f(0)$$
so the kernel is $1$-dimensional. Indeed the product $\prod_j(1-q_j)$ is the projection $P$ used to define the space of spinors $\cS=\cl(\ft\times\ft^*)P$, and hence $\prod_j(1-q_j)f(0)$ lies in the $1$-dimensional space $P\cS=P\cl(\ft\times\ft^*)P$ which has even grading.
\end{proof}

We have show that $(L^2(\ft)\tensor \cS,\rho,Q_0)$ defines an unbounded Kasparov triple. To show that it is an element of $KK_W(C(T)\otimes C(T^\vee),\C)$ it remains to consider the $W$-equivariance.

We begin by considering the abstract setup of a finite dimensional vector space $V$ equipped with the natural action of $\GL(V)$. This induces a diagonal action on $V\otimes V^*$.

If $V$ is equipped with a non-degenerate symmetric bilinear form $g$ then we can form the Clifford algebra $\cl(V)$. The subgroup $\O(g)$ of $\GL(V)$, consisting of those elements preserving $g$, acts naturally on $\cl(V)$. The bilinear form additionally gives an isomorphism from $V$ to $V^*$ and hence induces a bilinear form $g^*$ on $V^*$, allowing us to form the Clifford algebra $\cl(V^*)$. Clearly the dual action of $\O(g)$ on $V^*$ preserves $g^*$ hence there is a diagonal action of $\O(g)$ on $\cl(V)\tensor \cl(V^*)$ which we identify with $\cl(V\times V^*)$.

We say that an element $a$ of $\cl(V\times V^*)$ is \emph{symmetric} if there exists a $g$-orthonormal\footnote{We say that $\{\e_j\}$ is $g$-orthonormal if $g_{jk}=\pm\delta_{jk}$ for each $j,k$.} basis $\{\e_j: j=1,\dots,n\}$ with dual basis $\{\eps^j: j=1,\dots,n\}$ such that $a$ can be written as $p(\e_1\eps^1,\dots,\e_n\eps^n)$ where $p(x_1,\dots,x_n)$ is a symmetric polynomial.

\begin{proposition}\label{invariance of symmetric elements}
For any basis $\{\e_j\}$ of $V$ with dual basis $\{\eps^j\}$ for $V^*$, the Einstein sum $\e_j\otimes \eps^j$ in $V\otimes V^*$ is $\GL(V)$-invariant.

Suppose moreover that $V$ is equipped with a non-degenerate symmetric bilinear form $g$ and that the underlying field has characteristic zero. Then every symmetric element of $\cl(V)\tensor \cl(V^*)\cong \cl(V\times V^*)$ is $\O(g)$-invariant.
\end{proposition}
\begin{proof}
Identifying $V\otimes V^*$ with endomorphisms of $V$ in the natural way, the action of $\GL(V)$ is the action by conjugation and $\e_j\otimes \eps^j$ is the identity which is invariant under conjugation.

For the second part, over a field of characteristic zero the symmetric polynomials are generated by power sum symmetric polynomials $p(x_1,\dots,x_n)=x_1^k+\dots+x_n^k$, so it suffices to consider
\begin{align*}
p(\e_1\eps^1,\dots,\e_n\eps^n)&=(\e_1\eps^1)^k+\dots+(\e_n\eps^n)^k\\
&=(-1)^{k(k-1)/2}\Bigl((\e_1)^k(\eps^1)^k+\dots+(\e_n)^k(\eps^n)^k\Bigr).
\end{align*}
When $k$ is even, writing $(\e_j)^k=(\e_j^2)^{k/2}=(g_{jj})^{k/2}$ and similarly $(\eps^j)^k=(g^{jj})^{k/2}$, we see that each term $(\e_j)^k(\eps^j)^k$ is $1$ since $g_{jj}=g^{jj}=\pm1$ for an orthonormal basis. Thus $p(\e_1\eps^1,\dots,\e_n\eps^n)=n(-1)^{k(k-1)/2}$ which is invariant.

Similarly when $k$ is odd we get $(\e_j)^k(\eps^j)^k=\e_j\eps^j$ so
$$p(\e_1\eps^1,\dots,\e_n\eps^n)=(-1)^{k(k-1)/2}(\e_1\eps^1+\dots+\e_n\eps^n).$$
As the sum $\e_j\otimes \eps^j$ in $V\otimes V^*$ is invariant under $\GL(V)$, it is in particular invariant under $\O(g)$, and hence the sum $\e_j\eps^j$ is $\O(g)$-invariant in the Clifford algebra.
\end{proof}

\bigskip

Returning to our construction, the projection $P$ is a symmetric element of the Clifford algebra and hence is $W$-invariant by Proposition \ref{invariance of symmetric elements}. It follows that $\cS$ carries a representation of $W$. The space $L^2(\ft)$ also carries a representation of $W$ given by the action of $W$ on $\ft$ and we equip $L^2(\ft)\tensor \cS$ with the diagonal action of $W$.

To verify that the representation $\rho$ is $W$-equivariant it suffices to consider the representations of $C(T)$ and $C(T^\vee)$ separately. As the exponential map $\ft\to T$ is $W$-equivariant it is clear that the representation of $C(T)$ on $L^2(\ft)$ by pointwise multiplication is $W$-equivariant.

For $e^{2\pi i\la \eta,\gamma\ra}\in C(T^\vee)$ we have $w\cdot(e^{2\pi i\la \eta,\gamma\ra})=e^{2\pi i\la w^{-1}\cdot\eta,\gamma\ra}=e^{2\pi i\la \eta,w\cdot\gamma\ra}$ thus $\rho(w\cdot(e^{2\pi i\la \eta,\gamma\ra}))=L_{w\cdot\gamma}\otimes 1_\cS=L_{w}L_{\gamma}L_{w^{-1}}\otimes 1_\cS$. Thus the representation of $C(T^\vee)$ is also $W$-equivariant.

It remains to check that the operator $Q_0$ is $W$-equivariant. By definition
$$Q_0=\frac{\partial}{\partial y^j}\otimes \eps^j-2\pi i y^j\otimes \e_j.$$
Now by Proposition \ref{invariance of symmetric elements} $\frac{\partial}{\partial y^j}\otimes \eps^j=\e_j\otimes \eps^j$ is a $\GL(\ft)$-invariant element of $\ft\otimes \ft^*$ and so in particular it is $W$-invariant. Writing $y^j=\la \eps^j, y\ra$ the $W$-invariance of the second term again follows from invariance of $\e_j\otimes \eps^j$.

Hence we conclude the following.

\begin{theorem}
The triple $(L^2(\ft)\tensor \cS,\rho,Q_0)$ constructed above defines an element $\fa$ of $KK_W(C(T^\vee)\tensor C(T),\C)$.
\end{theorem}

\subsection{The Kasparov product $\fb\otimes_{C(T^\vee)}\fa$} We will compute the Kasparov product of $\fb\in KK_W(\C,C(T)\tensor C(T^\vee))$ with $\fa\in KK_W(C(T^\vee)\tensor C(T),\C)$ where the product is taken over $C(T^\vee)$ (not $C(T)\tensor C(T^\vee)$).

Recall that $\fb$ is given by the Kasparov triple $(\cE,1,0)$ where $\cE$ is the completion of $C_c(\ft)$ with the inner product
$$\la \phi_1,\phi_2\ra(x,\eta)=\sum_{\alpha,\beta \in\Gamma} \overline{\phi_1(x-\alpha)}\phi_2(x-\beta)e^{2\pi i \la \eta,\beta-\alpha\ra}$$
in $C(T)\tensor C(T^\vee)$. As above $\fa$ is given by the triple $(L^2(\ft)\tensor \cS,\rho,Q_0)$.

To form the Kasparov product we must take that tensor product of $\cE$ with $L^2(\ft)\tensor \cS$ over $C(T^\vee)$ and as the operator in the first triple is zero, the operator required for the Kasparov product can be any connection for $Q_0$.

We note that the representation $\rho$ is the identity on $\cS$ and hence
$$\cE\tensor_{C(T^\vee)}(L^2(\ft)\tensor \cS)=(\cE\tensor_{C(T^\vee)}L^2(\ft))\tensor \cS.$$
Thus we can focus on identifying the tensor product $\cE\tensor_{C(T^\vee)}L^2(\ft)$. By abuse of notation we will also let $\rho$ denote the representation of $C(T)\tensor C(T^\vee)$ on $L^2(\ft)$.

As we are taking the tensor product over $C(T^\vee)$, not over $C(T)\tensor C(T^\vee)$, we are forming the Hilbert module
$$(\cE\tensor C(T))\tensor_{C(T)\tensor C(T^\vee)\tensor C(T)}(C(T)\tensor L^2(\ft))$$
however since the algebra $C(T)$ is unital, it suffices to consider elementary tensors of the form $(\phi\otimes 1)\otimes (1\otimes \xi)$. Where there is no risk of confusion we will abbreviate these are $\phi\otimes \xi$

Let $\phi_1,\phi_2\in C_c(\ft)$ and let $\xi_1,\xi_2$ be elements of $L^2(\ft)$.  Then
\begin{align*}
\la \phi_1\otimes\xi_1,\phi_2\otimes\xi_2\ra &=\la 1\otimes \xi_1,(1\otimes\rho)(\la \phi_1,\phi_2\ra\otimes 1)(1\otimes \xi_2)\ra.
\end{align*}
The operator $(1\otimes\rho)(\la \phi_1,\phi_2\ra\otimes 1)$ corresponds to a field of operators
\begin{align*}(1\otimes\rho)(\la \phi_1,\phi_2\ra\otimes 1)(x)&=\sum_{\alpha,\beta \in\Gamma} \overline{\phi_1(x-\alpha)}\phi_2(x-\beta)\otimes \rho(e^{2\pi i \la \eta,\beta-\alpha\ra}\otimes 1)\\
&=\sum_{\alpha,\beta \in\Gamma} \overline{\phi_1(x-\alpha)}\phi_2(x-\beta)\otimes L_\alpha^*L_\beta
\end{align*}
and so
\begin{align*}
\la \phi_1\otimes\xi_1,\phi_2\otimes\xi_2\ra(x) &=\sum_{\alpha,\beta \in\Gamma} \overline{\phi_1(x-\alpha)}\phi_2(x-\beta)\la L_\alpha\xi_1, L_\beta\xi_2\ra\\
&=\la \sum_{\alpha\in \Gamma}\phi_1(x-\alpha)L_\alpha\xi_1,\sum_{\beta\in \Gamma}\phi_2(x-\beta) L_\beta\xi_2\ra.
\end{align*}
We note that $x\mapsto \sum_{\alpha\in \Gamma}\phi_1(x-\alpha)L_\alpha\xi_1$ is a continuous $\Gamma$-equivariant (and hence bounded) function from $\ft$ to $L^2(\ft)$. Let $C(\ft,L^2(\ft))^\Gamma$ denote the space of such functions equipped with the $C(T)$ module structure of pointwise multiplication in the first variable and gives the pointwise inner product $\la g_1,g_2\ra(x)=\la g_1(x),g_2(x)\ra$. We remark that equivariance implies this inner product is a $\Gamma$-periodic function on $\ft$.

The above calculation show that $\cE\tensor_{C(T^\vee)}L^2(\ft)$ maps isometrically into $C(\ft,L^2(\ft))^\Gamma$ via the map
$$\phi\otimes\xi \mapsto \sum_{\alpha\in \Gamma}\phi(x-\alpha)L_\alpha\xi.$$
Moreover this map is surjective.  To see this, note that if $\phi$ is supported inside a single fundamental domain then for $x$ in that fundamental domain we obtain the function $\phi(x)\xi$.  This is extended by equivariance to a function on $\ft$, and using a partition of unity one can approximate an arbitrary element of $C(\ft,L^2(\ft))^\Gamma$ by sums of functions of this form.

We now remark that $C(\ft,L^2(\ft))^\Gamma$ is in fact isomorphic to the Hilbert module $C(T, L^2(\ft))$ via a change of variables. Given $g\in C(\ft,L^2(\ft))^\Gamma$, let $\tilde h(x)=L_{-x}g(x)$. The $\Gamma$-equivariance of $g$ ensures that $g(\gamma+x)=L_\gamma g(x)$ whence
$$\tilde h(\gamma+x)=L_{-x-\gamma}g(\gamma+x)=L_{-x-\gamma}L_\gamma g(x)=L_{-x}g(x)=\tilde h(x).$$
As $\tilde h$ is a $\Gamma$-periodic function from $\ft$ to $L^2(\ft)$ we identify it via the exponential map with the continuous function $h$ from $T$ to $L^2(\ft)$ such that $\tilde h(x)=h(\exp(x))$. Hence $g\mapsto h$ defines the isomorphism $C(\ft,L^2(\ft))^\Gamma\cong C(T,L^2(\ft))$.

We now state the following theorem.

\begin{theorem}\label{Hilbert modules isomorphism}
The Hilbert module $\cE\tensor_{C(T^\vee)}(L^2(\ft)\tensor \cS)$ is isomorphic to $C(T,L^2(\ft)\tensor \cS)$ via the map
$$\phi\otimes (\xi\otimes s) \mapsto \sum_{\alpha\in \Gamma}\phi(x-\alpha)L_{\alpha-x}\xi\otimes s.$$
The representation of $C(T)$ on $L^2(\ft)$ induces a representation $\sigma$ of $C(T)$ on $C(T,L^2(\ft)\tensor \cS)$ defined by
$$[\sigma(f)h](\exp(x),y)=f(\exp(x+y))h(\exp(x),y).$$
Here the notation $h(\exp(x),y)$ denotes the value at the point $y\in\ft$ of $h(\exp(x))\in L^2(\ft)\tensor \cS$.
\end{theorem}

\begin{proof}
We recall that $\cE\tensor_{C(T^\vee)}(L^2(\ft)\tensor \cS)$ is isomorphic to $(\cE\tensor_{C(T^\vee)}L^2(\ft))\tensor \cS$ and we have established that $\cE\tensor_{C(T^\vee)}L^2(\ft)\cong C(T,L^2(\ft))$. This provides the claimed isomorphism.

It remains to identify the representation.  Given $f\in C(T)$ let $\tilde f(x)=f(\exp(x))$ denote the corresponding periodic function on $\ft$. By definition the representation of $C(T)$ on $\cE\tensor_{C(T^\vee)}(L^2(\ft)\tensor \cS)$ takes $\phi\otimes \xi\otimes s$ to $\phi\otimes \tilde f\xi\otimes s$. This is mapped under the isomorphism to the $\Gamma$-periodic function on $\ft$ whose value at $x$ is
$$\sum_{\alpha\in \Gamma}\phi(x-\alpha)L_{\alpha-x}(\tilde f\xi)\otimes s \in L^2(\ft)\tensor \cS.$$
Evaluating this element of $L^2(\ft)\tensor \cS$ at a point $y\in\ft$ we have 
$$\sum_{\alpha\in \Gamma}\phi(x-\alpha)\tilde f(x-\alpha+y)\xi (x-\alpha+y)\otimes s=\tilde f(x+y)\sum_{\alpha\in \Gamma}\phi(x-\alpha)[L_{\alpha-x}\xi](y)\otimes s$$
by $\Gamma$-periodicity of $\tilde f$. Thus $\sigma(f)$ pointwise multiplies the image of $\phi\otimes \xi\otimes s$ in $C(T,L^2(\ft)\tensor \cS)$ by $\tilde f(x+y)=f(\exp(x+y))$ as claimed.
\end{proof}

We now define an operator $\bQ$ on $C(T,L^2(\ft)\tensor \cS)$ by
$$(\bQ h)(\exp(x))=Q_0(h(\exp(x)))$$
for $h\in C(T,L^2(\ft)\tensor \cS)$.

\begin{theorem}
The unbounded operator $\bQ$ is a connection for $Q_0$ in the sense that the bounded operator $\bF=\bQ(1+\bQ^2)^{-1/2}$ is a connection for $F_0=Q_0(1+Q_0^2)^{-1/2}$, after making the identification of Hilbert modules as in Theorem \ref{Hilbert modules isomorphism}.
\end{theorem}

\begin{proof}
Let $Q_x=(L_x\otimes 1_\cS)Q_0(L_{-x}\otimes 1_\cS)$ and correspondingly define
$$F_x=Q_x(1+Q_x^2)^{-1/2}=(L_x\otimes 1_\cS)F_0(L_{-x}\otimes 1_\cS).$$
The commutators $[L_x\otimes 1_\cS,Q_0]$ are bounded (the argument is exactly as for $[L_\gamma\otimes 1_\cS,Q_0]$ in Section \ref{K-homology element}). It follows (in the spirit of Baaj-Julg, \cite{BJ}) that the commutators $[L_x\otimes 1_\cS,F_0]$ are compact. Thus $F_x-F_0$ is a compact operator for all $x\in \ft$.

To show that $\bF$ is a connection for $F_0$ we must show that for $\phi\in \cE$, the diagram
$$\begin{CD}
L^2(\ft)\tensor \cS @>F_0>> L^2(\ft)\tensor \cS\\
@V\phi\otimes VV @V\phi\otimes VV\\
\cE\otimes L^2(\ft)\tensor \cS @. \cE\otimes L^2(\ft)\tensor \cS\\
@V\cong VV @V\cong VV\\
C(T,L^2(\ft)\tensor \cS) @>>\bF> C(T,L^2(\ft)\tensor \cS)
\end{CD}$$
commutes modulo compact operators.

Following the diagram around the right-hand side we have 
$$\xi\otimes s \mapsto \sum_{\alpha\in\Gamma}\phi(x-\alpha)(L_{\alpha-x}\otimes 1_\cS)F_0(\xi\otimes s)$$
while following the left-hand side we have
$$\bF\Big[\sum_{\alpha\in\Gamma}\phi(x-\alpha)(L_{\alpha-x}\otimes 1_\cS)(\xi\otimes s)\Big]=\sum_{\alpha\in\Gamma}\phi(x-\alpha)F_0(L_{\alpha-x}\otimes 1_\cS)(\xi\otimes s).$$
As $[F_0,L_{\alpha-x}\otimes 1_\cS]$ is a compact operator for each $x$ and the sum is finite for each $x$, the difference between the two paths around the diagram is a function from $T$ to compact operators on $L^2(\ft)\tensor\cS$. It is thus a compact operator from the Hilbert space $L^2(\ft)\tensor\cS$ to the Hilbert module $C(T,L^2(\ft)\tensor \cS)$ as required.
\end{proof}

\begin{theorem}\label{first Kasparov product}
The Kasparov product $\fb\otimes_{C(T^\vee)}\fa$ is $1_{C(T)}$ in $KK_W(C(T),C(T))$.
\end{theorem}

\begin{proof}
We define a homotopy of representations of $C(T)$ on $C(T,L^2(\ft)\tensor \cS)$ by
$$[\sigma_\lambda(f)h](\exp(x),y)=f(\exp(x+\lambda y))h(\exp(x),y)$$
and note that $\sigma_1=\sigma$ while $\sigma_0$ is simply the representation of $C(T)$ on $C(T,L^2(\ft)\tensor \cS)$ by pointwise multiplication of functions on $T$. It is easy to see that these representations are $W$-equivariant.

Let $f$ be a smooth function on $T$ and let $h\in C(T,L^2(\ft)\tensor \cS)$. Let $\tilde{f}(x)=f(\exp(x))$ and let $\tilde{h}(x,y)=h(\exp(x),y)$.  Then
\begin{align*}
([\bQ,\sigma_\lambda(f)]h)&(\exp(x),y)\\
=\,\,&\big[\frac{\partial}{\partial y^j}(\eps^j \tilde f(x+\lambda y)\tilde h(x,y)) - 2\pi i y^j\e_j\tilde f(x+\lambda y)\tilde h(x,y)\big] \\
&-\big[\tilde f(x+\lambda y)\frac{\partial}{\partial y^j}(\eps^j \tilde h(x,y)) - \tilde f(x+\lambda y)2\pi i y^j\e_j\tilde h(x,y)\big]\\
=\,\,&\frac{\partial}{\partial y^j}(\tilde f(x+\lambda y))(\eps^j \tilde h(x,y)).
\end{align*}
For each $\lambda$ the operator $\bQ$ thus commutes with the representation $\sigma_\lambda$ modulo bounded operators on a dense subalgebra of $C(T)$. Hence for each $\lambda$ $(C(T,L^2(\ft)\tensor \cS),\sigma_\lambda,\bQ)$ defines an unbounded Kasparov triple.

This is true in particular for $\lambda=1$ and thus $(C(T,L^2(\ft)\tensor \cS),\sigma,\bQ)$ is a Kasparov triple so as the operator in the triple $\fb$ is zero while $\bQ$ is a connection for $Q_0$ it follows that $\fb\otimes_{C(T^\vee)} \fa=(C(T,L^2(\ft)\tensor \cS),\sigma,\bQ)$ in $KK_W(C(T),C(T))$.

Now applying the homotopy we have $\fb\otimes_{C(T^\vee)} \fa=(C(T,L^2(\ft)\tensor \cS),\sigma_0,\bQ)$. Since $\sigma_0$ commutes exactly with the operator $\bQ$ the representation$\sigma_0$ respects the direct sum decomposition of $C(T,L^2(\ft)\tensor \cS)$ as $C(T,\ker(Q_0))\oplus C(T,\ker(Q_0)^\perp)$. The operator $\bQ$ is invertible on the second summand (and commutes with the representation) and hence the corresponding Kasparov triple $(C(T,\ker(Q_0)^\perp),\sigma_0|_{C(T,\ker(Q_0)^\perp)},\bQ|_{C(T,\ker(Q_0)^\perp)})$ is zero in $KK$-theory.

We thus conclude that $\fb\otimes_{C(T^\vee)} \fa=(C(T,\ker(Q_0)),\sigma_0|_{C(T,\ker(Q_0))},0)$. Since $\ker Q_0$ is 1-dimensional (Theorem \ref{ladder}) the module $C(T,\ker(Q_0))$ is isomorphic to $C(T)$ and the restriction of $\sigma_0$ to this is the identity representation of $C(T)$ on itself.  Thus $\fb\otimes_{C(T^\vee)} \fa=(C(T),1,0)=1_{C(T)}$.
\end{proof}

\subsection{The Kasparov product $\fb\otimes_{C(T)}\fa$}

We begin by considering the Langlands dual picture, which exchanges the roles of $T$ and $T^\vee$. There exist elements $\fa^\vee \in KK_W(C(T)\tensor C(T^\vee),\C)$ and $\fb^\vee\in KK_W(\C,C(T^\vee)\tensor C(T))$ for which the result of the previous section implies $\fb^\vee\otimes_{C(T)}\fa^\vee=1_{C(T^\vee)}$ in $KK_W(C(T^\vee),C(T^\vee))$.

We will show that there is an isomorphism $\theta:C(T^\vee)\tensor C(T)\to C(T)\tensor C(T^\vee)$ such that $\fa=\theta^*\fa^\vee$ and $\fb=\theta^{-1}_*\fb^\vee$. This will imply that $\fb\otimes_{C(T)}\fa=\fb^\vee\otimes_{C(T)}\fa^\vee=1_{C(T^\vee)}$ in $KK_W(C(T^\vee),C(T^\vee))$ and hence will complete the proof of the Poincar\'e duality between $C(T)$ and $C(T^\vee)$.

We recall that $\fa$ is represented by the unbounded Kasparov triple $(L^2(\ft)\tensor \cS,\rho,Q_0)$ where $\cS=\cl(\ft\times\ft^*)P$, for $P$ the projection $P=\prod_{j}\frac 12(1-i\e_j\eps^j)$ and
$$Q_0=\frac{\partial}{\partial y^j}\otimes \eps^j-2\pi i y^j\otimes \e_j.$$
For $\gamma\in \Gamma$, $\chi\in\Gamma^\vee$ and correspondingly $e^{2\pi i\la \eta,\gamma\ra}$ in $C(T^\vee)$, $e^{2\pi i\la \chi,x\ra }$ in $C(T)$, the representation $\rho$ of $C(T^\vee)\tensor C(T)$ is defined by
$$\rho(e^{2\pi i \la \eta,\gamma\ra})(\xi\otimes s)=L_\gamma\xi\otimes s, \text { and }\rho(e^{2\pi i\la \chi,x\ra})(\xi\otimes s)=e^{2\pi i\la \chi,x\ra}\xi \otimes s.$$

By definition $\fa^\vee$ is represented by the triple $(L^2(\ft^*)\tensor \cS^\vee,\rho^\vee,Q_0^\vee)$ where $\cS^\vee=\cl(\ft^*\times\ft)P^\vee$, for $P^\vee$ the projection $P^\vee=\prod_{j}\frac 12(1-i\eps^j\e_j)$ and
$$Q_0^\vee=\frac{\partial}{\partial \eta_j}\otimes \e_j-2\pi i \eta_j\otimes \eps^j.$$
For $\gamma\in \Gamma$, $\chi\in\Gamma^\vee$ and correspondingly $e^{2\pi i\la \eta,\gamma\ra}$ in $C(T^\vee)$, $e^{2\pi i\la \chi,x\ra }$ in $C(T)$, the representation $\rho^\vee$ of $C(T)\tensor C(T^\vee)$ is now defined by
$$\rho^\vee(e^{2\pi i\la \chi,x\ra})(\xi^\vee\otimes s^\vee)=L^\vee_\chi\xi^\vee \otimes s^\vee, \text { and }\rho^\vee(e^{2\pi i \la \eta,\gamma\ra})(\xi^\vee\otimes s^\vee)=e^{2\pi i \la \eta,\gamma\ra}\xi^\vee\otimes s^\vee.$$
Here $L^\vee_\chi$ denotes the translation action of $\chi\in\Gamma^\vee$ on $L^2(\ft^*)$.

In our notation, $\eps^j$ is again an orthonormal basis for $\ft^*$ and $\e_j$ is an orthonormal basis for $\ft$. We can canonically identify $\cl(\ft\times\ft^*)$ with $\cl(\ft^*\times\ft)$, and hence think of both $\cS$ and $\cS^\vee$ as subspaces of this algebra.

\bigskip

We can identify $L^2(\ft)$ with $L^2(\ft^*)$ via the Fourier transform: let $\cF:L^2(\ft)\to L^2(\ft^*)$ denote the Fourier transform isomorphism
$$[\cF\xi](\eta)=\int_\ft \xi(y)e^{2\pi i\la \eta,y\ra}\,dy.$$
It is easy to see that this is $W$-equivariant.

To identify $\cS$ with $\cS^\vee$, let $u\in \cl(\ft\times\ft^*)$ be defined by $u=\eps^1\eps^2\dots\eps^n$ when $n=\dim(\ft)$ is even and $u=e_1e_2\dots e_n$ when $n$ is odd.

\begin{lemma}
Conjugation by $u$ defines a $W$-equivariant unitary isomorphism $\cU:\cS\to\cS^\vee$. For $a\in\cl(\ft\times\ft^*)$ (viewed as an operator on $\cS$ by Clifford multiplication) $\cU a\cU^*$ is Clifford multiplication by $uau^*$ on $\cS^\vee$ and in particular $\cU \e_j\cU^*=\e_j$ while $\cU \eps^j\cU^*=-\eps^j$.
\end{lemma}

\begin{proof}
We first note that $u$ respectively commutes and anticommutes with $\e_j$, $\eps^j$ (there being respectively an even or odd number of terms in $u$ which anticommute with $\e_j$, $\eps^j$). It follows that $uPu^*=P^\vee$, hence conjugation by $u$ maps $\cS$ to $\cS^\vee$.

Denoting by $\pi:\C P\to \C$ the identification of $\C P$ with $\C$, the inner product on $\cS$ is given by $\la s_1,s_2\ra= \pi(s_1^*s_2)$ while the inner product on $\cS^\vee$ is given by $\la s_1^\vee,s_2^\vee\ra= \pi(u^*(s_1^\vee)^*s_2^\vee u)$. Thus
$$\la usu^*,s^\vee\ra=\pi(u^*(usu^*)^*s^\vee u)=\pi(s^*u^*s^\vee u)=\la s,u^*s^\vee u\ra$$so $\cU^*$ is conjugation by $u^*$ which inverts $\cU$ establishing that $\cU$ is unitary.

We now check that $\cU$ is $W$-equivariant. In the case that $\ft$ is even-dimensional, we note that identifying $\cl(\ft^*)$ with the exterior algebra of $\ft^*$ (as a $W$-vector space), $u$ corresponds to the volume form on $\ft^*$ so $w\cdot u=\det (w)u$. Similarly in the odd dimensional case $u$ corresponds to the volume form on $\ft$ and again the action of $w$ on $u$ is multiplication by the determinant. Thus
$$w\cdot\cU(s)=w\cdot (usu^*)=(w\cdot u)(w\cdot s)(w\cdot u^*)=\det(w)^2\,u(w\cdot s)u^*=\cU(w\cdot s)$$
since $\det(w)=\pm1$.

Finally for $s^\vee\in\cS^\vee$ and $a\in\cl(\ft\times\ft^*)$ we have
$$\cU a\cU^*s^\vee=\cU(au^*s^\vee u)=uau^*s^\vee$$
and hence $\cU \e_j\cU^*=u\e_ju^*=\e_j$, $\cU \eps^j\cU^*=u\eps^ju^*=-\eps^j$.
\end{proof}

Since $\cF\otimes\cU$ is a $W$-equivariant unitary isomorphism from $L^2(\ft)\tensor \cS$ to $L^2(\ft^*)\tensor \cS^\vee$, the triple $(L^2(\ft)\tensor \cS,\rho,Q_0)$ representing $\fa$ is isomorphic to the Kasparov triple
$$(L^2(\ft^*)\tensor \cS^\vee,(\cF\otimes \cU)\rho(\cF^*\otimes \cU^*),(\cF\otimes u)Q_0(\cF^*\otimes \cU^*)).$$

\begin{theorem}\label{theta^* fa^vee}
Let $\theta:C(T^\vee)\tensor C(T)\to C(T)\tensor C(T^\vee)$ be defined by
$$\theta(g\otimes f)=f\otimes (g\circ \epsilon).$$
where $\epsilon$ is the involution on $T^\vee$ defined by $\epsilon(\exp(\eta))=\exp(-\eta)$. Then $\fa=\theta^*\fa^\vee$ in $KK_W(C(T^\vee)\tensor C(T),\C)$.
\end{theorem}

\begin{proof}
We will show that $\rho^\vee\circ\theta=(\cF\otimes \cU)\rho(\cF^*\otimes \cU^*)$ and $(\cF\otimes u)Q_0(\cF^*\otimes \cU^*)=Q_0^\vee$. We begin with the operator.

The operator $Q_0$ is given by
$$\frac{\partial}{\partial y^j}\otimes \eps^j-2\pi i y^j\otimes \e_j.$$
Conjugating the operator $\frac{\partial}{\partial y^j}$ by the Fourier transform we obtain the multiplication by $2\pi i\eta_j$, while conjugating $-2\pi i y^j$ by the Fourier transform we obtain the multiplication by $-2\pi i(\frac i{2\pi}\frac{\partial}{\partial \eta_j})=\frac{\partial}{\partial \eta_j}$. Conjugation by $\cU$ negates $\eps^j$ and preserves $\e_j$ hence
$$(\cF\otimes u)Q_0(\cF^*\otimes \cU^*)=2\pi i \eta_j\otimes (-\eps^j)+\frac{\partial}{\partial \eta_j}\otimes \e_j=Q_0^\vee.$$
For trhe representation, $\rho(e^{2\pi i\la \chi,x\ra})$ is multiplication by $e^{2\pi i\la \chi,x\ra}$ on $L^2(\ft)$ (with the identity on $\cS$) and conjugating by the Fourier transform we get the translation $L^\vee_\chi$, hence $(\cF\otimes \cU)\rho(e^{2\pi i\la \chi,x\ra})(\cF^*\otimes \cU^*)=\rho^\vee(e^{2\pi i\la \chi,x\ra})$.  On the other hand $\rho(e^{2\pi i \la \eta,\gamma\ra})$ is the translation $L_\gamma$ and Fourier transforming we get the multiplication by $e^{-2\pi i \la \eta,\gamma\ra}$. Thus  $(\cF\otimes \cU)\rho(e^{2\pi i \la \eta,\gamma\ra})(\cF^*\otimes \cU^*)=\rho^\vee(e^{2\pi i \la -\eta,\gamma\ra})$.

We conclude that $(\cF\otimes \cU)\rho(\cF^*\otimes \cU^*)=\rho^\vee\circ\,\theta$ as required.
\end{proof}

\begin{theorem}\label{second Kasparov product}
The Kasparov product $\fb\otimes_{C(T)}\fa$ is $1_{C(T^\vee)}$ in the Kasparov group $KK_W(C(T^\vee),C(T^\vee))$. Hence the elements $\fa\in KK_W(C(T^\vee)\tensor C(T),\C)$ and $\fb\in KK_W(\C,C(T)\tensor C(T^\vee))$ exhibit a $W$-equivariant Poincar\'e duality between the algebras $C(T)$ and $C(T^\vee)$.
\end{theorem}

\begin{proof}
We have $\fb\otimes_{C(T^\vee)}\fa=1_{C(T)}$ in $KK_W(C(T),C(T))$ by Theorem \ref{first Kasparov product} while  $\fb^\vee\otimes_{C(T)}\fa^\vee=1_{C(T^\vee)}$ in $KK_W(C(T^\vee),C(T^\vee))$ by  Theorem \ref{first Kasparov product} for the dual group.

By Theorem \ref{theta^* fa^vee} we have $\fa^\vee=(\theta^{-1})^*\fa$ whence
$$1_{C(T^\vee)}=\fb^\vee\otimes_{C(T)}\fa^\vee=(\theta^{-1})_*\fb^\vee\otimes_{C(T)} \fa.$$
Let $\fb'=(\theta^{-1})_*\fb^\vee$ in $KK_W(\C,C(T)\tensor C(T^\vee))$. Then
$$\fb=\fb\otimes_{C(T^\vee)}1_{C(T^\vee)}=\fb\otimes_{C(T^\vee)}(\fb'\otimes_{C(T)} \fa).$$
By definition $\fb'\otimes_{C(T)} \fa=(\fb'\otimes 1_{C(T^\vee)})\otimes_{_{\scriptstyle C(T)\otimes C(T^\vee)}} \fa$ \,and hence 
$$\fb=(\fb\otimes\fb')\otimes_{_{\scriptstyle C(T^\vee)\tensor C(T)}} \fa$$
by associativity of the Kasparov product. Here $\fb\otimes\fb'$ is the `external' product and lives in $KK_W(\C,C(T)\tensor C(T)\tensor C(T^\vee)\tensor C(T^\vee))$, with $\fb$ appearing in the first and last factors, and $\fb'$ in the second and third. The product with $\fa$ is over the second and last factors. Similarly
$$\fb'=\fb'\otimes_{C(T)}(\fb\otimes_{C(T^\vee)} \fa)=(\fb'\otimes\fb)\otimes_{_{\scriptstyle C(T)\tensor C(T^\vee)}} \fa$$
where $\fb'$ now appears as the first and last factors and the product with $\fa$ is over the first and third factors. Up to reordering terms of the tensor product $(\fb\otimes\fb')\otimes_{_{\scriptstyle C(T^\vee)\tensor C(T)}} \fa=(\fb'\otimes\fb)\otimes_{_{\scriptstyle C(T)\tensor C(T^\vee)}} \fa$.

Thus (by commutativity of the external product) $\fb=\fb'=(\theta^{-1})_*\fb^\vee$ and hence $\fb\otimes_{C(T)}\fa=1_{C(T^\vee)}$. This completes the proof.
\end{proof}

\begin{corollary}\label{PD}
The Kasparov product with $\fb=(\cE,1,0)$ induces an isomorphism from $KK_W(C(T),\C)$ to $KK_W(\C,C(T^\vee))$.
\end{corollary}

\section{Affine and Extended Affine Weyl groups}\label{affine Weyl section}

In this section we will give the precise definitions of the affine and extended affine Weyl groups of a compact connected semisimple Lie group. As noted earlier these are semidirect products of lattices in the Lie algebra $\ft$ of a maximal torus $T$ by the Weyl group $W$. The affine Weyl group $W_a$ is a Coxeter group while the extended affine Weyl group contains $W_a$ as a finite index normal subgroup.  The quotient $W'_a/W_a$ is the fundamental group of the Lie group $G$.

Recall that $\Gamma(T)$ is the kernel of the exponential map $\exp : \ft \to T$. Let $p:\widetilde G\to G$ denote the universal cover and let $\widetilde T$ be the preimage of $T$ which is a maximal torus in $\widetilde G$.
We consider the following commutative diagram.
$$\begin{CD}
@.\Gamma(\widetilde T)@>>>\ft @>>>\widetilde T@>>> 0\\
@. @VV\iota V @VV{\mathrm{id}}V @VV{p|_{\widetilde T}}V\\
0@>>>\Gamma(T)@>>>\ft @>>>T@.
\end{CD}$$
By the snake lemma the sequence
$$
\begin{matrix}
\ker(\mathrm{id})&\to& \ker(p|_{\widetilde T})&\to& \coker(\iota)&\to& \coker(\mathrm{id})\\
||&&||&&||&&||\\
0& &\pi_1(G)& & \Gamma(T)/\Gamma(\widetilde T)& &0
\end{matrix}$$
is exact, hence $\Gamma(T)/\Gamma(\widetilde T)$ is isomorphic to $\pi_1(G)$. We thus have a map from $\Gamma(T)$ onto $\pi_1(G)$.  The kernel of this map is denoted $N(G,T)$ however we have seen that this is the nodal lattice $\Gamma(\widetilde T)$ for the torus $\widetilde T$.

\begin{definition}
The \emph{affine Weyl group of $G$} is 
\[
W_a = W_a(G) = N(G,T) \rtimes W
\]
 and the \emph{extended
affine Weyl group of $G$} is 
\[
W'_a = W_a(G) = \Gamma(T) \rtimes W\]
where $W$ denotes the Weyl group of $G$.
\end{definition}

The following is now immediate.

\begin{lemma}\label{cover} Let $\widetilde{G}$ denote the universal cover of $G$ and let $\widetilde{T}$ denote
a maximal torus in $\widetilde{G}$.   Then we have
\[
W_a(G) = W'_a(\widetilde{G}) = W_a(\widetilde{G})
\]\qed
\end{lemma}

We remark that in general the extended affine Weyl group $W_a'(G)$ is a split extension of $W_a(G)$ by $\pi_1(G)$ i.e.\ a semidirect product.

\section{Langlands Duality and $K$-theory}\label{the end section}

In this section we will consider the $K$-theory of the affine and extended affine Weyl groups of a compact connected semisimple Lie group. We will prove the main results stated in the introduction.

\subsection{The proof of Theorem \ref{the main theorem}}

We begin with the following elementary result which applies to a general semidirect product.

\begin{lemma}\label{semidirect=crossed}
Let $\Gamma \rtimes W$ be a semidirect product of discrete groups and let $A$ be a $\Gamma \rtimes W$-\;$C^*$-algebra. Then
$$A\rtimes_r(\Gamma\rtimes W) \cong (A\rtimes_r\Gamma)\rtimes_rW.$$
\end{lemma}

\begin{proof}
It is easy to see that the obvious map from the twisted group ring $A[\Gamma\rtimes W] \to (A\rtimes_r\Gamma)\rtimes_rW$ is a homomorphism of $*$-algebras with dense image. To verify that the completions are isomorphic one simply notes that both completions are defined by representing the algebras as operators on a Hilbert space $H\otimes \ell^2(\Gamma)\otimes \ell^2(W)$ where $A$ is faithfully represented on $H$.
\end{proof}

We will now consider the left-hand side of the assembly map. The following theorem generalises the familiar identification $KK^*_{\Gamma}(C_0(\underline E\Gamma),\C)\simeq K_*(B\Gamma)$ for a torsion-free group $\Gamma$.

\def\EGammaW{\underline E}
\begin{theorem}\label{LHS}
Let $\Gamma$ be a torsion-free group and $W$ a finite group acting by automorphisms on $\Gamma$. Let $Z$ be any proper cocompact $\Gamma\rtimes W$-space. Then
$$KK^*_{\Gamma\rtimes W}(C_0(Z),\C)\cong KK^*_W(C(Z/\Gamma),\C).$$
Taking the direct limit over all $\Gamma\rtimes W$-compact subspaces $Z$ of a universal example for proper actions $\EGammaW(\Gamma\rtimes W)$, the left-hand side of the Baum-Connes assembly map for $\Gamma\rtimes W$ is
$$\lim_{\to} KK^*_W(C(Z/\Gamma),\C)$$
and in particular if $B\Gamma=\EGammaW(\Gamma\rtimes W)/\Gamma$ is compact then the left-hand side of the assembly map is $KK^*_W(C(B\Gamma),\C)$.
\end{theorem}

\begin{remark}\label{LHS remark}
For the extended affine Weyl group, $W_a'=\Gamma\rtimes W$, the Lie algebra $\ft$ of a maximal torus $T$ for $G$ provides a universal example for proper actions of $W_a'$. Thus the left-hand side of the Baum-Connes assembly map for $W_a'$ is
$$KK^*_{W_a'}(C_0(\ft),\C)\simeq KK^*_W(C(T),\C).$$
\end{remark}

\begin{proof}[Proof of Theorem \ref{LHS}]
By the Green-Julg theorem \cite[Theorem 20.2.7(b)]{Black} we have
$$KK^*_{\Gamma\rtimes W}(C_0(Z),\C)\simeq KK^*(C_0(Z)\rtimes_r (\Gamma\rtimes W),\C).$$
By Lemma \ref{semidirect=crossed} $C_0(Z)\rtimes_r (\Gamma\rtimes W)\simeq (C_0(Z)\rtimes_r \Gamma)\rtimes_r W$, hence applying the Green-Julg theorem again we have
$$KK^*(C_0(Z)\rtimes_r (\Gamma\rtimes W),\C)\simeq KK^*_W(C_0(Z)\rtimes_r \Gamma,\C).$$
Finally $C_0(Z)\rtimes_r \Gamma$ is $W$-equivariantly Morita equivalent to $C(Z/\Gamma)$, hence
$$KK^*_W(C_0(Z)\rtimes_r \Gamma,\C)\simeq KK^*_W(C(Z/\Gamma),\C).$$
The result now follows.
\end{proof}

We now move on to the right-hand side of the assembly map.

\begin{theorem}\label{RHS} Let $G$ be a compact connected semisimple Lie group with extended affine Weyl group $W_a'=\Gamma\rtimes W$.
\begin{enumerate}\renewcommand{\theenumi}{\alph{enumi}}
\item The group $C^*$-algebra $C^*_r(W'_a)$ is isomorphic to $(C(T^\vee)\otimes \cB(\ell^2(W)))^W$ where $T^\vee$ is a maximal torus of $G^\vee$, and $W$ acts diagonally on the tensor product.

\item The right-hand side of the Baum-Connes assembly map for $W_a'$, $KK^*(\C,C^*_r(W'_a))$, is isomorphic to $KK^*_W(\C,C(T^\vee))$.
\end{enumerate}
\end{theorem}

\begin{proof}
By Lemma \ref{semidirect=crossed}
\begin{align*}
C^*_r(W'_a(G)) =   C^*_r(\Gamma(T) \rtimes W) \simeq C^*_r(\Gamma(T)) \rtimes_r W.
\end{align*}
This is isomorphic to $C(T^\vee)\rtimes_r W$ by Lemma \ref{gammadual}. By Lemma \ref{crossed product lemma} this is isomorphic to $(C(T^\vee)\otimes \cB(\ell^2(W)))^W$ establishing (a).

At the level of $K$-theory we have
$$KK(\C,C^*_r(W'_a))\simeq KK(\C,C(T) \rtimes W)\simeq KK_W(\C,C(T^\vee))$$
by the Green-Julg theorem \cite[Theorem 20.2.7(a)]{Black} establishing (b).
\end{proof}

Our first main result now follows.

\begin{proof}[Proof of Theorem \ref{the main theorem}]
This follows from Theorem \ref{LHS}, Remark \ref{LHS remark}, Theorem \ref{RHS} and Corollary \ref{PD}.
\end{proof}

\subsection{$K$-theory isomorphisms for affine and extended affine Weyl groups}
Recall that in Section \ref{triangle section} we considered the extended affine Weyl groups of $\PSU_3$ and its Langlands dual $\SU_3$.  The extended affine Weyl group of the latter is the affine Weyl group for both of these Lie groups.  We saw that although the two extended affine Weyl groups are non-isomorphic, their group $C^*$-algebras have the same $K$-theory.

In this section we will show that this is not a coincidence, indeed passing to the Langlands dual always rationally preserves the $K$-theory for the extended affine Weyl groups. In particular where the extended affine Weyl group of the dual of $G$ agrees with the affine Weyl group of $G$ (as for $\PSU_3$) the $K$-theory for the affine and extended affine Weyl groups of $G$ agrees up to rational isomorphism.
 
\begin{theoremn}
\dualthm
\end{theoremn}
\medskip

\begin{proof}
By Theorem \ref{RHS} we have an isomorphism
\begin{align}\label{K1}
K_*(C^*_r(W_a'(G)))\simeq K^W_*(C(T^\vee))=K_W^*(T^\vee).
\end{align}
Dually, applying the proposition to $G^\vee$ we have
\begin{align}\label{K2}
K_*(C^*_r(W_a'(G^\vee)))\simeq K^W_*(C(T))=K_W^*(T).
\end{align}

Now by Corollary \ref{PD} there is a Poincar\'e duality isomorphism
\begin{align}\label{K3}
K^*_W(T^{\vee}) & \simeq  K_*^W(T).
\end{align}

Applying the universal coefficient theorem, we have the exact sequence
\[
0 \to \Ext^1_{\Z}(K^{*-1}_W(T), \Z) \to K_*^W(T) \to \Hom(K^*_W(T),\Z) \to 0
\]
In particular the torsion-free part of $K_*^W(T)$ agrees with the torsion-free part of $K^*_W(T)$ therefore rationally we have
\begin{align}\label{K4}
K_*^W(T) \simeq K_W^*(T).
\end{align}

The theorem now follows by combining (\ref{K1}), (\ref{K2}), (\ref{K3}), (\ref{K4}).
\end{proof}

\begin{corollaryn}
\affinecor
\end{corollaryn}

\begin{proof}
If $G$ is a compact connected semisimple Lie group of adjoint type then its Langlands dual $G^\vee$ is simply connected so $W_a'(G^\vee)=W_a(G^\vee)$.

In the case that $G$ is additionally of type $A_n, D_n, E_6, E_7, E_8, F_4, G_2$ the group $G^\vee$ is the universal cover of $G$ and hence $W_a(G)=W_a(G^\vee)=W_a'(G^\vee)$.
\end{proof}

\end{document}